\documentclass[10pt,a4paper]{article}
\usepackage{amsmath,amssymb,amsfonts,amsthm}
\usepackage{graphicx}
\usepackage{float,graphicx,color}
\usepackage[all]{xy}

\DeclareMathOperator{\vol}{vol}

\newcommand{\rmd}{\mathrm{d}}

\newcommand{\rmL}{\mathrm{L}}

\newcommand{\bbN}{\mathbb{N}}

\newcommand{\bbR}{\mathbb{R}}

\newcommand{\frE}{\mathfrak{E}}

\newcommand{\frS}{\mathfrak{S}}

\newcommand{\frV}{\mathfrak{V}}

\newcommand{\calB}{\mathcal{B}}
\newcommand{\calC}{\mathcal{C}}

\newcommand{\calF}{\mathcal{F}}
\newcommand{\calG}{\mathcal{G}}

\newcommand{\calP}{\mathcal{P}}
\newcommand{\calQ}{\mathcal{Q}}

\newcommand{\calT}{\mathcal{T}}
\newcommand{\calU}{\mathcal{U}}
\newcommand{\calV}{\mathcal{V}}
\newcommand{\calX}{\mathcal{X}}
\newcommand{\calY}{\mathcal{Y}}
\newcommand{\calZ}{\mathcal{Z}}

\newcommand{\beq}{\begin{equation}}
\newcommand{\eeq}{\end{equation}}

\newcommand{\ts}{\textstyle}

\newcommand{\Alt}{\rmL_{a}}

\newcommand{\WF}{\mathfrak{W}}

\newcommand{\bs}{{\scriptscriptstyle \bullet}}

\newcommand{\rmins}[1]{ {\textrm{ #1 } } }

\newcommand{\myand}{ \quad \textrm{and} \quad }

\newcommand{\for}{\textrm{ for }}
\newcommand{\myfor}{\for}

\newcommand{\orient}{o}
\newcommand{\deltat}{\delta'} 
\newcommand{\deltal}{\tau}

\newcommand{\subcell}{\lhd} 

\newcommand{\lspan}{\mathrm{span}}

\newcommand{\ctr}{\lfloor}

\newcommand{\poly}{\calP}
\newcommand{\alt}{\Lambda}
\newcommand{\simp}{\calC}

\newcommand{\mapping}[4]{
\left\{
\begin{array}{rcl}
\displaystyle #1  &\to& #2 ,\\
\displaystyle #3  &\mapsto      & #4.
\end{array} \right.
}

\newtheorem{theorem}{Theorem}[section]
\newtheorem{lemma}[theorem]{Lemma}
\newtheorem{corollary}[theorem]{Corollary}
\newtheorem{proposition}[theorem]{Proposition}

\theoremstyle{definition}
\newtheorem{definition}{Definition}[section]

\theoremstyle{remark}
\newtheorem{remark}{Remark}[section]

\theoremstyle{plain}

 \title{On high order finite element spaces\\ of differential forms}

\author{Snorre H. {\sc Christiansen}\footnote{Department of Mathematics, University of Oslo, PO box 1053 Blindern, NO-0316 Oslo, Norway. email: {\tt snorrec@math.uio.no} (corresponding author).} \and Francesca {\sc Rapetti}\footnote{Universit\'e Nice Sophia Antipolis, CNRS,  LJAD, UMR 7351, 06100 Nice, France. }}

\date{}
\begin{document}

\maketitle

\begin{abstract}
We show how the high order finite element spaces of differential forms due to Raviart-Thomas-N\'edelec-Hiptmair fit into the framework of finite element systems, in an elaboration of the finite element exterior calculus of Arnold-Falk-Winther. Based on observations by Bossavit, we provide new low order degrees of freedom. As an alternative to existing choices of bases, we provide canonical resolutions in terms of scalar polynomials and Whitney forms.
\end{abstract}


\bigskip


\smallskip



\section{Introduction}

Mixed finite elements, adapted to the fundamental differential operators gradient, curl and divergence,  were introduced in \cite{RavTho77} for $\bbR^2$ and \cite{Ned80} for $\bbR^3$. They have developed into a powerful tool for simulating a wide range of partial differential equations modelling for instance fluid flow and electromagnetic waves \cite{BreFor91}\cite{RobTho91}\cite{Mon03}. As pointed out in \cite{Bos88}\cite{Bos98}, lowest order mixed finite elements correspond to constructs in algebraic topology known as Whitney forms \cite{Whi57}. In \cite{Wei52} they are actually attributed to de Rham (see also the footnote p. 139 in \cite{Whi57}). Their initial purpose was to relate the de Rham sequence of smooth differential forms to simplicial cochain sequences, and prove de Rham's theorem, that these sequences have isomorphic cohomology groups. Eigenvalue convergence for corresponding discretizations of the Hodge Laplacian has also been proved \cite{DodPat76}, prefiguring some of the results obtained in a finite element context, for which we refer to \cite{Bof10}\cite{ChrWin13IMA}. Mixed finite elements extending Whitney forms to high order -- that is, higher degree polynomials -- and N\'ed\'elec elements to any space dimension, were presented in \cite{Hip99}\cite{Hip02}. 
Giving a lead role to differential complexes in numerical analysis by
developing the interplay between differential topology and finite element methods, has given rise to the subject of \emph{finite element exterior calculus}, programmed in \cite{Arn02} and most recently reviewed in \cite{ArnFalWin10}.

We refer to the high order spaces  of \cite{Hip99} as \emph{trimmed} polynomial differential forms, since for a given degree, some of the top degree polynomials are (carefully) removed from the space. These spaces are naturally spanned by products of polynomials with Whitney forms, but this is not a tensor product, as we explain later. Nevertheless, as remarked in \cite{Chr07NM}\footnote{The preprint, available on the arXiv, contained a number of results excluded from the version published by Numer. Math., but of importance to the present paper.}, they have a filtering by polynomial degree compatible with the wedge product, the corresponding de Rham theorem (concerning cohomology groups) follows from the existence of particular degrees of freedom, and eigenvalue convergence for the Hodge Laplacian follows from composing the resulting interpolators with a local smoothing operator. In \cite{ArnFalWin06} the construction of \cite{Hip99} was reworked in terms of the Koszul complex, emphasizing the fundamental duality, through degrees of freedom, between the there denominated $\poly^-_r\alt^k$ and $\poly_r\alt^k$ spaces, which correspond to N\'ed\'elec's first \cite{Ned80} and second \cite{Ned86} family. It was also placed in a general framework of discretizations of Hilbert complexes by subcomplexes, which has applications to other situations, such as elasticity.

A powerful tool for the convergence analysis of discretizations of such complexes is commuting projections which are stable in appropriate norms. Such projections have been constructed for the de Rham complex of differential forms equipped with the $\rmL^2$ metric, \cite{Sch08}\cite{Chr07NM}\cite{ArnFalWin06}\cite{ChrWin08}\cite{ChrMunOwr11}. One recovers in particular,  in this degree of generality, convergence for the eigenvalues of discretizations of the Hodge Laplacian. For a discussion of the existence of commuting projections, see \cite{ChrWin13IMA}. In \cite{ChrMunOwr11}, the techniques are extended to include $\rmL^p$ estimates, with applications to discrete Sobolev injections and translation estimates.

A notion of \emph{finite element system} (FES) has also been developed, see \cite{Chr08M3AS}\cite{Chr09AWM} and especially \cite{ChrMunOwr11}(\S 5). It is a general theory, abstracting the good properties of known mixed finite elements, but  allowing for cellular decompositions of space and non-polynomial differential forms. In precisely determined circumstances, they provide good subcomplexes of the de Rham complex. This framework can  be used to construct dual finite elements \cite{BufChr07}\cite{Chr08M3AS}, minimal ones \cite{Chr10CRAS}, tensor-products \cite{Chr09AWM}\cite{ChrMunOwr11} and upwinded variants \cite{Chr13FoCM}. 

That trimmed polynomial differential forms fit nicely into the framework of FESs can be deduced directly from results in \cite{Hip99}\cite{ArnFalWin06}. However the general theorems of FESs allows one to streamline the proofs. Our first task in the present work is to spell this out. For this purpose, we provide some new results both on FESs and on trimmed polynomial differential forms. In particular, we rely on a general way of checking an extension property of element systems (Proposition \ref{prop:extrec}), and give a new result on the duality between $\poly^-_r\alt^k_0$ and  $\poly_r\alt^k$ (Propositions \ref{prop:diagdom}). 

Taking the cue from \cite{RapBos09}, we provide new degrees of freedom (dofs) which, for $k$-forms, consist in integrating on some "small" $k$-simplexes, and which, for scalar polynomials, result in  the Lagrange basis. These degrees of freedom are overdetermining (Proposition \ref{prop:overdet}). We also provide a general way to deduce unisolvent degrees of freedom from overdetermining ones, in such a way that the associated interpolator commutes.

As already mentioned, trimmed polynomial differential forms can be generated by taking products of polynomials and Whitney forms. Canonical bases of scalar polynomials in barycentric coordinates (such as the Lagrange and Bernstein bases) and the canonical basis of Whitney forms  provide canonical spanning families.  However these are not free, and in this sense the product is not a tensor product (see Remark \ref{rem:nottensor}). Many recipes for extracting a basis have been proposed. In \cite{ArnFalWin09}, bases are proposed for differential forms in arbitrary space dimension, and these have been used for implementations \cite{Kir14}. See also the references in \cite{ArnFalWin09}\cite{RapBos09} for previous constructions in the case of space dimension two and three. We mention in particular \cite{AinCoy03}\cite{GopGarDem05}\cite{SchZag05}. 

However,  it seems impossible to get a canonical basis: they all depend on the numbering of vertices. Typically, some members are just removed from the natural spanning families. In this paper we take a different approach, providing a systematic treatment of the linear relations in canonical spanning families, in terms of resolutions. Thus we provide canonical resolutions of $\poly^-_r\alt^k$ and $\poly^-_r\alt^k_0$, in which all the spaces have a canonical basis (Propositions \ref{prop:respl0} and \ref{prop:respl}).

Finally we provide some details on how to compute with canonical spanning families. We construct a parametrized family of bases of polynomials on a simplex, which contains the Lagrange and the Bernstein basis as special cases. We compute scalar products of differential forms from the data consisting of edge lengths in the mesh. We also provide a formula for computing wedge products. Finally we make some remarks on tensor products.

The paper is organized as follows. In \S \ref{sec:framework} we introduce the notion of FES, and discuss conditions under which it is \emph{compatible} in the sense, for instance, of admitting a commuting interpolator.  In \S \ref{sec:trimmed} we show how trimmed polynomial differential forms fit into that framework. In particular we prove compatibility. We also introduce a notion of "small" degrees of freedom, and construct canonical resolutions equipped with canonical bases. In \S \ref{sec:various} we provide explicit formulas for a number of operations in barycentric coordinates.

\section{Abstract framework\label{sec:framework}}

\subsection{Glossary}
The abstract framework of Finite Element Systems involves a number of concepts that are either new or given a more precise meaning than usual. The following list with references is provided for the convenience of the reader:

\begin{itemize}
\item \emph{cell, cellular complex:} Definition \ref{def:cell}.
\item \emph{subcomplex, subcell, $\subcell (T)$:} Equation (\ref{eq:subc}) and above.
\item \emph{boundary of a cell:} Equation (\ref{eq:boundary}).
\item \emph{$\calT^k$:} Equation (\ref{eq:caltk}).
\item \emph{relative orientation, $\orient(T',T)$:} after Equation (\ref{eq:boundary}).
\item \emph{cochains, coboundary:} after Equation (\ref{eq:caltk}).
\item \emph{element system:} Definition \ref{def:es}.
\item \emph{$E^k(\calT)$:} Equation (\ref{eq:invlim}).
\item \emph{element systems: extension property, local exactness and compatibility:} Definition \ref{def:extloc}.
\item \emph{$E^k_0(T)$:} Definition \ref{def:extloc}.
\item \emph{system of degrees of freedom, sysdof:} Definition \ref{def:sysdof}.
\item \emph{unisolvence of a sysdof:} Equation (\ref{eq:freedom}).
\item \emph{interpolator:} Definition \ref{def:int}.
\item \emph{harmonic degrees of freedom, harmonic interpolator:} Equations (\ref{eq:harmdofk},\ref{eq:harmdof}).
\end{itemize}

\subsection{Discrete and differential geometry \label{sec:discdiffgeo}}

\begin{definition}\label{def:cell}
A \emph{cell} in a metric space $S$ is a subset either reduced to a singleton, or for which there is a Lipschitz isomorphism (a bijection which is Lipschitz in both directions) to the closed unit ball in $\bbR^k$ ($k \in \bbN$, $k \neq 0$). The uniquely determined $k \in \bbN$ is called the dimension of the cell.  The interior and the boundary of a cell are those inherited from the corresponding unit ball (they do not depend on the choice of Lipschitz isomorphism). By convention, a singleton has dimension $0$ and has empty boundary.

A \emph{cellular complex} for a metric space $S$ is a collection $\calT$ of cells in $S$, such that the following conditions hold:
\begin{itemize}
\item Distinct cells in $\calT$ have disjoint interiors.
\item The boundary of any cell in $\calT$ is a union of cells in $\calT$.
\item The union of all cells in $\calT$ is $S$.
\end{itemize}
\end{definition}

A cellular \emph{subcomplex} of $\calT$ is a subset of $\calT$ which is a cellular complex (for a subspace of $S$). A subcell of a cell $T$ is an element $T'$ of $\calT$ included in $T$. We write $T' \subcell T$ describe this situation. For $T \in \calT$ its subcells consitute a cellular subcomplex denoted $\subcell (T)$. Thus, with our notations:
\begin{equation}\label{eq:subc}
\subcell (T) = \{T' \ : \ T' \subcell T\}.
\end{equation}
The boundary of a cell is denoted $\partial T$ and considered equipped with the cellular subcomplex:
\begin{equation}\label{eq:boundary}
\partial T = \{T' \in \subcell (T) \ : \ T' \neq T \}.
\end{equation}

In the following we suppose that each cell $T \in \calT$ of dimension at least $1$ has been oriented (as a manifold). The \emph{relative orientation} of two cells $T$ and $T'$ in $\calT$, also called the \emph{incidence number}, is denoted $\orient(T,T')$. It is non-zero only when $T'$ is in the boundary of $T$ and has codimension $1$, in which case its value  is $\pm 1$, depending on whether $T'$ is outward oriented compared with $T$. This definition guarantees that the following Stokes theorem holds. For any smooth enough $k$-form $u$ on the $(k+1)$-cell $T$, we have:
\begin{equation}\label{eq:stokes}
\int_{T} \rmd u = \sum_{T' \subcell T} \orient(T, T') \int_{T'} u.
\end{equation}
When the cell $T$ has dimension $1$, the relative orientation $\orient(T, T')$ of its vertices, is defined so that the one dimensional Stokes theorem holds, integration on a point being function evaluation.

We let $\calT^k$ denote the subset of $\calT$ consisting of cells of dimension $k$:
\begin{equation}\label{eq:caltk}
\calT^k = \{ T \in \calT \ : \ \dim T = k \}.
\end{equation} 

For each $k$, maps $c:\calT^k \to \bbR$ are called $k$-\emph{cochains}, and they constitute a vector space denoted $\calC^k(\calT)$. The \emph{coboundary} operator $\delta : \calC^k(\calT) \to \calC^{k+1}(\calT)$ is defined by:
\begin{equation}
(\delta c)_T= \sum_{T'\in \calT^{k+1}} \orient(T,T')c_{T'}.
\end{equation}
The adjoint of $\delta$ is the boundary operator, denoted $\delta'$. We have $\delta \delta = 0$ so that the family $\calC^\bs(\calT)$ is a complex, called the cochain complex and represented by:
\beq
\xymatrix{
0 \ar[r] & \calC^0(\calT) \ar[r]^{\delta} & \calC^1(\calT) \ar[r]^{\delta} & \cdots
}
\eeq

When $S$ is a smooth manifold we denote by $\Omega^k(S)$ the space of smooth differential $k$-forms on $S$. 
For each $k$ we denote by $\rho^k : \Omega^k(S) \to \calC^k(\calT)$ the de Rham map,  which is defined by:
\begin{equation}
\rho^k: u \mapsto (\int_T u )_{T\in \calT^k}.
\end{equation}
As an application of Stokes theorem  on the cells of dimension $k+1$, in the form (\ref{eq:stokes}), the following diagram commutes:
\begin{equation}
\xymatrix{
\Omega^k(S) \ar[r]^\rmd \ar[d]^{\rho^k} & \Omega^{k+1}(S) \ar[d]^{\rho^{k+1}}\\
\calC^k(\calT) \ar[r]^{\delta} & \calC^{k+1}(\calT)
}
\end{equation}
A celebrated theorem of de Rham states that the above morphism of complexes induces isomorphisms on cohomology groups. Whitney forms, which will be introduced later, provide a tool for proving this \cite{Wei52}\cite{Whi57}.

\subsection{Finite element systems}

If $T$ is a cell in a cellular complex $\cal T$,  we denote by $\frE^k(T)$ the set of $k$-forms on $T$ with the following property:  for any $T' \in \calT$ included in $T$,  its pullback to $T'$ is in $\rmL^2(T')$ and has its exterior derivative in $\rmL^2(T')$.

\begin{definition}\label{def:es}
Suppose $\calT$ is a cellular complex. An \emph{element system} on $\calT$, is a family of  closed subspaces $E^k(T)  \subseteq \frE^k(T)$, one for each $k\in \bbN$ and each $T \in \calT$, subject to the following requirements:
\begin{itemize}
\item The exterior derivative should induce maps:
\beq
\rmd: E^k(T) \to E^{k+1}(T).
\eeq
\item If $T' \subseteq T$ are two cells in $\calT$ and  $i_{TT'}: T' \to T$ denotes the canonical injection, then pullback by $i_{TT'}$ should induce a map:
\beq
i_{TT'}^\star: E^k(T) \to E^k(T').
\eeq
\end{itemize}
\end{definition}
For instance the spaces $\frE^k(T)$ constitute an element system. A finite element system is one in which all the spaces are finite dimensional.

We define $E^k(\calT)$ as follows :
\begin{equation}\label{eq:invlim}
E^k(\calT)= \{u \in \bigoplus_{T \in \calT} E^k(T) \ : \  \forall T, T' \in \calT \quad  T' \subseteq T \Rightarrow u_{T}|_{T'} = u_{T'}\}. 
\end{equation}
In this definition $u_{T}|_{T'}$ denotes the pullback of $u_T$ to $T'$ by the inclusion map. In terms of category theory this is an inverse limit.
 
Elements of $E^k(\calT)$ may be regarded as differential forms defined piece-wise, which are continuous across interfaces between cells, in the sense of equal pullbacks to the interface. In $\bbR^3$  one-forms and two-forms correspond to vector fields, and then the continuity holds for tangential and normal components respectively.

This definition will also be applied to cellular subcomplexes of $\calT$, such as the boundary $\partial T$ of any given cell $T\in \calT$. In particular this is the meaning given to $E^k(\partial T)$ in the following definition.

\begin{definition}\label{def:extloc}
Consider now the following two conditions on an element system $E$ on a cellular complex $\calT$: 
\begin{itemize}
\item \emph{Extensions.} For each $T\in \calT$ and $k\in \bbN$, the restriction operator(pullback to the boundary) $E^k(T) \to E^k(\partial T)$ is onto. The kernel of this map is denoted $E^k_0(T)$.
\item \emph{Local exactness.} The following sequence is exact for each $T\in \calT$:
\begin{equation}\label{eq:coh}
\xymatrix{
0 \ar[r] & \bbR \ar[r] & E^0(T) \ar[r]^\rmd & E^1(T) \ar[r]^\rmd & \cdots \ar[r]^\rmd &  E^{\dim T}(T) \ar[r] & 0.
}
\end{equation}
The second arrow sends an element of $ \bbR$ to the constant function on $T$ taking this value.
\end{itemize}
We will say that an element system \emph{admits extensions} if the first condition  holds, is \emph{locally exact} if the second condition holds and is \emph{compatible} if both hold.
\end{definition}

The following result strengthens Proposition 3.1 in \cite{Chr08M3AS} by an "only if" (see also Proposition 5.14 in \cite{ChrMunOwr11}).
\begin{proposition}\label{prop:extdim}
Let $E$ be a FES on a cellular complex $\calT$. Then: 
\begin{itemize}
\item We have:
\begin{equation}\label{eq:bounddim}
\dim E^k (\calT) \leq \sum_{ T\in \calT}\dim E^k_0(T).
\end{equation}
\item Equality holds in (\ref{eq:bounddim})  if and only if $E$ admits extensions for $k$-forms on each $T \in \calT$.
\end{itemize}
\end{proposition}
\begin{proof} We denote by $\calT^{[m]}$ the $m$-skeleton of $\calT$, namely the cellular subcomplex consisting of cells of dimension at most $m$. Thus:
\beq
\calT^{[m]} = \calT^0 \cup \ldots \cup \calT^m.
\eeq

(i) Consider the complex:
\begin{equation}\label{eq:restrskel}
\xymatrix{
0 \ar[r] & \bigoplus_{T \in \calT^m} E^k_0(T) \ar[r] & E^k(\calT^{[m]}) \ar[r] & E^k(\calT^{[m-1]} )\ar[r] & 0,
}
\end{equation}
where the second arrow is inclusion, and the third is restriction.

We get:
\begin{equation}\label{eq:skeldim}
\dim E^k(\calT^{[m]}) \leq \sum_{T \in \calT^m} \dim E^k_0(T) + \dim E^k(\calT^{[m-1]}).
\end{equation}

From this, (\ref{eq:bounddim}) follows.

(ii) If $E$ admits extensions for $k$-forms, (\ref{eq:restrskel}) is an exact sequence, so that equality holds in (\ref{eq:skeldim}). Hence equality holds in (\ref{eq:bounddim}).

(iii) Suppose  equality holds in (\ref{eq:bounddim}). 

We claim first that for all $m$-skeletons of $\calT$:
\[ 
\dim E^k (\calT^{[m]}) =  \sum_{ T\in \calT^{[m]}}\dim  E^k_0(T).
\] 
Indeed, if the claim is true for $m$, then we deduce from (\ref{eq:skeldim}):
\[ 
\dim E^k(\calT^{[m-1]}) \geq \sum_{ T\in \calT^{[m]}}\dim E^k_0(T) - \sum_{T \in \calT^m} \dim E^k_0(T) = \sum_{ T\in \calT^{[m-1]}}\dim  E^k_0(T).
\] 
Equality follows, and from this the claim must also be true for $m-1$.

This proves the claim for all $m$, by a downward induction.

By dimension count, it then follows that each sequence (\ref{eq:restrskel}) is exact.

Now consider a cell $T$ and $u \in E^k(\partial T)$. Define $u$ on $k$-cells in $\calT$. Extend it to $(k+1)$-cells using exactness of (\ref{eq:restrskel}) for $m = k+1$. Continue extending $u$, increasing $m$ by one at each step. Whenever $u$ was already defined on a cell, just keep the old value. This gives an extension of $u$ in $E^k(\calT)$, hence also in  $E^k(T)$.
\end{proof}

The following is a useful way of checking the extension property, abstracting the technique of proof of Proposition 3.3 in \cite{Chr07NM}(preprint). See also the related notion of consistent extension operators of \cite{ArnFalWin09} (in particular Theorem 4.3).

\begin{proposition}\label{prop:extrec} Suppose that $E$ is an element system and that $U \in \calT$. Suppose that, for each cell $V \in \partial U$, each element $v$ of $E^k_0(V)$ can be extended to an element $u = e_V(v)$ of $E^k(U)$ in such a way that ($u|_V = v$ and) for each cell $V' \in \partial U$ with the same dimension as $V$, but different from $V$, we have $u |_{V'} = 0$. Then $E^k$ admits extensions on $U$.
\end{proposition}
\begin{proof}
Pick $v \in E^k(\partial U)$.  Define $u_{k-1}= 0 \in E^k( U)$.

 Pick $l \geq k-1$ and suppose that we have a $u_{l} \in E^k(U)$ such that $v - u_l |_{\partial U}$ is $0$ on all $l$-dimensional cells in $\partial U$. Put $w_l = v - u_l |_{\partial U}$. For each $(l+1)$-dimensional cell $V$ in $\partial U$, extend $w_l |_{V}$ to an element $e_V (w_l |_{V}) \in E^k(U)$, such that:
\begin{align}
(e_V (w_l |_V)) |_{V^{\phantom{\prime}}} & =  w_l |_V,\nonumber \\
(e_V (w_l |_V)) |_{V'} & = 0  \myfor V' \neq V,\ \dim V' = \dim V =  l+1.\nonumber 
\end{align}
Then put:
\[ 
u_{l+1} = u_l + \sum_{V \ : \ \dim V = l+1} e_V (w_l |_V) .
\]
Then $v - u_{l+1} |_{\partial U}$ is $0$ on all $(l+1)$-dimensional cells in $\partial U$. 

We may repeat until $l+1 = \dim U$ and then $u_{l+1}$ is the required extension of $v$.
\end{proof}

Recall Proposition 5.17 in \cite{ChrMunOwr11}:
\begin{proposition}\label{prop:equiv}
For an element system with extensions the exactness of (\ref{eq:coh}) on each $T\in \calT$ is equivalent to the combination of the following two conditions:
\begin{itemize}
\item For each $T \in \calT$,  $E^0(T)$ contains the constant functions.
\item For each $T\in \calT$, the following sequence (with boundary condition) is exact:
\begin{equation} \label{eq:coh0}
\xymatrix{
0 \ar[r] & E^0_0(T) \ar[r] & E^1_0(T) \ar[r] & \cdots \ar[r] & E^{\dim T}(T) \ar[r] & \bbR \ar[r] & 0.
}
\end{equation}
The second to last arrow is integration.
\end{itemize}
\end{proposition}

Recall the definition of $\frE^k(T)$ given at the beginning of this section.

\begin{proposition}\label{prop:ecomp}
The spaces $\frE^k(T)$ constitute a compatible element system.
\end{proposition}
\begin{proof}
The Bogovskii integral operator \cite{CosMcI10} can be used to prove the exactness of the sequences (\ref{eq:coh0}) on reference balls. Then the result is transported to cells by the chosen Lipschitz isomorphism. 

The extension property follows from the fact that any form on $\partial T$ which is $\rmL^2$ with exterior derivative in $\rmL^2$ can be extended to a form on $T$ with similar regularity. Then, starting with a form in $\frE^k(\partial T)$, one obtains an extension which is trivially in $\frE^k(T)$.
\end{proof}

\subsection{Degrees of freedom and interpolators \label{sec:dofint}}

The following is a formalization of the notion of dofs, with particular emphasis on their geometric location.
\begin{definition} \label{def:sysdof}
Given a cellular complex $\calT$, a \emph{system of degrees of freedom} (sysdof) is a choice, for each $k$ and $T$, of a subspace $\calF^k(T)$ of $\frE^k(T)^\star$,
\end{definition}




We remark that if we have cells $T' \subseteq T$ elements of $\frE^k(T')^\star$ can be considered as elements of $\frE^k(T)^\star$, by pullback.

Any $k$-form $u$ in $\frE^k(T)$ then gives a linear form $\langle \cdot , u \rangle$ on $\calF^k(T)$.  We put:
\begin{equation}
\Phi^k(T) u = \langle \cdot, u|_{T' }\rangle_{T' \in \calT} \in \bigoplus_{T' \subcell T} \calF^k(T')^\star.
\end{equation}

We say that a system of degrees of freedom is \emph{unisolvent} on $E$ if, for any $T \in \calT$, $\Phi^k(T)$ provides an isomorphism:
\begin{equation}\label{eq:freedom}
\Phi^k(T) : E^k(T) \to  \bigoplus_{T' \subcell T} \calF^k(T')^\star.
\end{equation}
This notion can be expressed informally by saying that elements of $E^k(T)$ are uniquely determined by (the values) the degrees of freedom (take on them). We stress that just as we, in general do not dispose of a canonical basis for the spaces constituting our element systems, we de not have a canonical basis for the spaces constituting the sysdofs.

The following strengthens Proposition 5.35 (and parts of 5.37) in \cite{ChrMunOwr11}. See also how, in \cite{ArnFalWin06},  Lemma 4.7 and Theorem 4.9 lead to Theorem 4.10.

\begin{proposition}\label{prop:unisolve}
Suppose that $E$ is a finite element system and that $\calF$ is a system of degrees of freedom. The following statements  are equivalent:
\begin{itemize} 
\item $\calF$ is unisolvent on $E$.
\item$E$ has the extension property and for each cell $T\in \calT$, the map $E^k_0(T) \to \calF^k(T)^\star$ is an isomorphism.
\item For each cell $T\in \calT$, the map $E^k_0(T) \to \calF^k(T)^\star$ is injective and
\begin{equation}\label{eq:dimf}
\dim E^k(T) = \sum_{T' \subcell T} \dim \calF^k(T').
\end{equation}
\end{itemize}
\end{proposition}
\begin{proof}
(i) Suppose the first condition holds. Then if $u \in E^k(\partial T)$ we can extend it to $T$ by imposing $0$ as degree of freedom on $T'=T$  in (\ref{eq:freedom}). So $E$ has the extension property.

We also get injectivity of the map:
\begin{equation}\label{eq:locuni}
E^k_0(T) \to  \calF^k(T)^\star.
\end{equation}
Moreover any element of the right hand side provides an element $ u\in E^k(T)$ with  degree of freedom $0$ on the cells of boundary. By unisolvence on the boundary cells we get $ u\in E^k_0(T)$. This shows that (\ref{eq:locuni}) is bijective. 
So the second condition holds.

(ii) If the second condition holds, the third follows from Proposition \ref{prop:extdim}.

(iii) Suppose the third condition holds. Then $\Phi^k(T)$ defined in (\ref{eq:freedom}) is injective between spaces of the same dimension, therefore $\Phi^k(T)$ is an isomorphism and the first condition holds.
\end{proof}

\begin{definition}\label{def:int}
For a finite element system $E$, an \emph{interpolator} is a collection of projection operators $I^k(T) : \frE^k(T) \to E^k(T)$, one for each $k \in \bbN$ and $T \in \calT$, which commute with restrictions to sub-cells. That is, whenever $T' \subcell T$ the following diagrams, in which horizontal maps are pullbacks, should commute:
\begin{equation}
\xymatrix{
\frE^k(T) \ar[r] \ar[d]^{I^k(T)} & \frE^k(T')\ar[d]^{I^k(T')}\\
 E^k(T) \ar[r] &  E^k(T')
}
\end{equation}
\end{definition}

One can then denote it simply with $I^\bs$ and extend it unambiguously to any sub-complex $\calT'$ of $\calT$. Any unisolvent system of degrees of freedom defines an interpolator by requiring, for each $T \in \calT$:
\begin{equation}
\Phi^k(T) I^k(T) u = \Phi^k(T) u.
\end{equation}
We will refer to it as the interpolator associated with the system of degrees of freedom.

The following is Proposition 5.37 in \cite{ChrMunOwr11}.
\begin{proposition}The following statements are equivalent:
\begin{itemize}
\item $E$ admits extensions,
\item $E$ has a unisolvent system of degrees of freedom,
\item $E$ can be equipped with an interpolator.
\end{itemize}
\end{proposition}

In practice one is interested in commuting interpolators. The following result reproduces Proposition 5.41 in \cite{ChrMunOwr11}.
\begin{proposition}\label{prop:com}
If $E$ is compatible and  $\calF$ a unisolvent system of degrees of freedom, the associated interpolator commutes with the exterior derivative if and only if:
\begin{equation}\label{eq:dadjoint}
\forall l \in \calF^k(T) \quad  l \circ \rmd \in \bigoplus_{T' \subcell T} \calF^k(T').
\end{equation}
\end{proposition}

Suppose that $\calF$ is a system of degrees of freedom on $\calT$ such that (\ref{eq:dadjoint}) holds. Then we have a well defined map $\hat \rmd: l\mapsto l \circ \rmd $ from $\calF^k(\calT)$ to $\calF^{k-1}(\calT)$, with:
\begin{equation}
 \calF^k(\calT) = \bigoplus_{T \in \calT} \calF^k(T) \subseteq \frE^k(\calT)^\star.
\end{equation}
Denote by $\delta$ its adjoint, which maps from $\calF^{k-1}(\calT)^\star$ to $\calF^k(\calT)^\star$.
The following diagram commutes:
\begin{equation}
\xymatrix{
\frE^{k-1}(\calT) \ar[r]^\rmd \ar[d]^{\Phi^{k-1}} & \frE^{k}(\calT) \ar[d]^{\Phi^k}\\
\calF^{k-1}(\calT)^\star \ar[r]^\delta & \calF^k(\calT)^\star
}
\end{equation}

Equip each $\frE^k(T)$ with a continuous scalar product $a$, for instance the $\rmL^2$ product on forms. For a given finite element system $E$, define spaces $\calF^k(T)$ as follows. For  $k = \dim T$:
\begin{align} 
\calF^k(T) & = \{a( \cdot, v) \ : v \in \rmd E^{k-1}_0(T)\} \oplus \{ \bbR \ts \int \cdot \}, \label{eq:harmdofk}
\end{align}
and for $k < \dim T$:
\begin{align}
\calF^k(T) & = \{a( \cdot, v) \ : v \in \rmd E^{k-1}_0(T)\} \oplus \{ a( \rmd \cdot, v) \ : v \in \rmd E^{k}_0(T)\}. \label{eq:harmdof}
\end{align} 
This is the natural generalization, to the adopted setting, of projection based interpolation, as defined in \cite{DemBab03}\cite{DemBuf05}. We call these the \emph{ harmonic degrees of freedom}, and the associated interpolator is called the \emph{harmonic interpolator}. 

\begin{remark}
Comparing with \cite{SchZag05}\cite{ArnFalWin09} we notice that these degrees of freedom provide a geometric decomposition of $E^\bs(\calT)$, in which the exterior derivative remains local, except for the second term in (\ref{eq:harmdofk}) which corresponds to the coboundary operator in $\calC^\bs(\calT)$.
\end{remark}

\begin{proposition}The following statements are equivalent:
\begin{itemize}
\item $E$ is compatible,
\item On $E$, the associated harmonic degrees of freedom are unisolvent,
\item $E$ has a unisolvent system of degrees of freedom, with property (\ref{eq:dadjoint}),
\item $E$ can be equipped with a commuting interpolator.
\end{itemize}
\end{proposition}
\begin{proof}
(i) If $E$ is compatible, the sequences (\ref{eq:coh0}) are exact, hence the harmonic degrees of freedom are unisolvent according to the second characterization in Proposition \ref{prop:unisolve}. So the first condition implies the second.

(ii) Harmonic degrees of freedom satisfy (\ref{eq:dadjoint}), so the second condition implies the third.

(iii) Unisolvent degrees of freedom with property  (\ref{eq:dadjoint}) provide a commuting interpolator according to Proposition \ref{prop:com}. So the third condition implies the fourth.

(iv) Suppose $E$ has a commuting interpolator. Recall Proposition \ref{prop:ecomp}. The commuting interpolator can be used to deduce compatibility of $E$ from the compatibility of $\frE$.  So the fourth condition implies the first.
\end{proof}

\subsection{Handling spanning families with resolutions}

High order finite element spaces of differential forms (which will be defined in the next sections) seem to lack natural choices of bases. But they do have natural spanning families. We wish to show how one can compute with these natural families, which requires a systematic handling of the linear relations that occur in them.

A caricature of our point is the following. Define:
\begin{equation}
V= \{ x \in \bbR^n \ : \ \sum_i x_i = 0 \}.
\end{equation}
Suppose we want to do linear algebra computations in $V$. Instead of making an arbitrary choice of basis for $V$, we compute in the canonical spanning family $e_i, (1 \leq i \leq n)$ defined by:
\begin{equation}
e_i = (0, \ldots, 0, 1,0, \ldots, 0) - \frac{1}{n}(1, \ldots, 1),
\end{equation}
where the first occurence of $1$ is in position $i$. We have to keep in mind that we have the relation:
\begin{equation}
\sum_i e_i = 0.
\end{equation}
This is the only one. We may summarize this by writing an exact sequence:
\begin{equation}
\xymatrix{
0 \ar[r] &\bbR \ar[r] & \bbR^n \ar[r] & V \ar[r] & 0,
}
\eeq
where the second arrow maps $1$ to $(1, \ldots, 1)$, and the third arrow maps the canonical basis to $(e_i)$. Any element $v$ of $V$ can be written:
\begin{equation}
v = \sum_i v_i e_i,
\end{equation}
and the $v_i \in \bbR$ are unique if we impose the condition:
\begin{equation}
\sum_i v_i = 0.
\end{equation}

Consider now the more general case where $V$ is some finite dimensional vector space (we have in mind the spaces constituting a FES). A spanning family $(e_i)_{i \in I}$ of $V$ corresponds to a surjective linear map $\epsilon:\bbR^I \to V$, namely the map sending the canonical basis of $\bbR^I$ to the given spanning family. We say that we have eliminated the redundancies in the spanning family if we have identified a matrix $(C_{ji})$ for $(j,i) \in J \times I$, with independent rows, such that for each $v \in V$ there is a unique $u \in \bbR^I$ verifying:
\begin{equation}
v = \sum_{i \in I} u_i e_i,
\end{equation}
and 
\begin{equation}
\sum_{i \in I} C_{ji}u_i = 0, \qquad \forall j \in J.
\end{equation}

Suppose we have a matrix $(B_{ij})$ for $(i,j) \in I \times J$, whose columns constitute a basis of $\ker \epsilon \subseteq \bbR^I$ indexed by $J$. This means that we have an exact sequence:
\beq
\xymatrix{
0 \ar[r] &\bbR^J \ar[r]^B & \bbR^I \ar[r]^\epsilon & V \ar[r] & 0.
}
\eeq
Any elimination of redundancies in the spanning family $(e_i)_{i \in I}$ is then equivalent to a choice of a matrix $C$ such that $CB$ is invertible, since its rows must represent linear forms on $\bbR^I$ whose restrictions to $\ker \epsilon$ constitute a basis of $(\ker \epsilon)^\star$.

The advantage of this approach is that, for the examples we have in mind, there are spanning families $(e_i)_{i \in I}$ which are natural. In particular, defining matrices for various linear operators with respect to such families is relatively easy. We then supply information pertaining to the natural spanning family, which will let the software construct the matrices $B$ and $C$.  This information will be canonical in a sense, and we let the non-canonical choices implicit in $B$ and $C$ be handled by the software. 

To determine a kernel matrix $B$ as above we rely on a more general notion, namely resolutions.
A finite resolution of a finite dimensional vector space $V$ is a sequence $W_0, \ldots, W_n$ of finite dimensional vector spaces, equipped with operators $e: W_0 \to V$ and $f_i : W_i \to W_{i-1}$ such that we have an exact sequence:
\begin{equation}
\xymatrix{
0 \ar[r] & W_n \ar[r]^{f_n} &  \ldots \ar[r]^{f_2} & W_1 \ar[r]^{f_1} & W_0 \ar[r]^{e} & V \ar[r] & 0.
}
\end{equation}

In the context of finite elements, it can occur that one has a space $V$ without a particular natural basis, but with a finite resolution as above, where each space $W_i$  has a natural basis (note that there can be several natural bases). Choosing one natural basis for each $W_i$, say indexed by a set $I_i$, we get a new resolution:
\begin{equation}
\xymatrix{
0 \ar[r] & \bbR^{I_n} \ar[r]^{F_n} &  \ldots \ar[r]^{F_2} & \bbR^{I_1} \ar[r]^{F_1} & \bbR^{I_0} \ar[r]^{\epsilon} & V \ar[r] & 0.
}
\end{equation}
In particular, one has a spanning family $(e_i)_{i \in {I}}$, with $I= I_0$, for $V$, such that the columns of the matrix $F_1$ span the linear relations. Determining a basis of the columns of $F_1$ is a problem of linear algebra that is easily handed over to the computer. 

Thus, for the finite element systems $E$ we will consider, we shall provide canonical resolutions of the spaces $E^k(T)$ and $E^k_0(T)$.

By Proposition \ref{prop:unisolve}, unisolvent degrees of freedom provide an isomorphism:
\begin{equation}
E^k(\calT) \to \bigoplus_{T \in \calT} \calF^k(T)^\star.
\end{equation}
and also, one, for each $T \in \calT$:
\begin{equation}
E^k_0(T) \to  \calF^k(T)^\star.
\end{equation}
They combine to an isomorphism:
\begin{equation}\label{eq:geodec}
\bigoplus_{T \in \calT} E^k_0(T) \to E^k(\calT).
\end{equation}
In particular if we have spanning families for the $E^k_0(T)$ where we have control of linear relations as above, we get spanning families for $E^k(\calT)$ with good control of the linear relations. Locally the isomorphism (\ref{eq:geodec}) induces isomorphisms:
\begin{equation}\label{eq:locgeodec}
\bigoplus_{T' \subcell T} E^k_0(T') \to E^k(T).
\end{equation}
This map can be expressed in the canonical spanning families we provide for $E^k(T)$ and $E^k_0(T')$. 

The maps (\ref{eq:geodec})(\ref{eq:locgeodec}) provide \emph{geometric decompositions} of $E^k(\calT)$ and $E^k(T)$. As in \cite{ArnFalWin09}, one might construct them directly from some variant of Proposition \ref{prop:extrec}, instead of going via degrees of freedom. 

Scalar products such as those appearing in mass and stiffness matrices are first expressed in the canonical spanning families of $E^k(T)$, when $T$ is a cell of maximal dimension in $\calT$. Then they are expressed in the spanning families of $ E^k_0(T')$, simply by composing with the matrix of the geometric decomposition.

With canonical spanning families thus implemented, we can let the computer determine a basis that optimizes some stability estimate. The best implementation could also involve Lagrange multipliers, as suggested by the above caricature. The exact form this should take, is beyond the scope of this paper. Our point is to provide a clean framework in which such juggling can be carried out by the computer, rather than to supply a particular choice of basis ready for implementation (such a basis, as far as we can see, would necessarily be non-canonical).   

\section{Trimmed polynomial differential forms\label{sec:trimmed}}

\subsection{Definitions}
We now turn to the application of the above framework to a specific example.

Let $\calT$ be a simplicial complex, spanning the domain $S$. The set of $k$-dimensional simplices in $\calT$ is denoted $\calT^k$. We suppose that the simplices are oriented. Recall that for $k\geq 1$, if $T$ is a $k$-simplex and $T'$ is a $(k-1)$-face of $T$, we denote by $\orient(T,T')$ their relative orientations. In this section a simplex is identified with a set of vertices (rather than their convex hull).

Following  \cite{Wei52}\cite{Whi57}, Whitney forms on $\calT$ can be defined as follows. 

For any vertex $i \in \calT^0$, $\lambda_i$ is the corresponding barycentric coordinate map, namely the continuous piecewise affine map taking the value $1$ at vertex $i$ and $0$ at the other vertices.

For a  $k \geq 1$, let $T\in \calT^k$ be a $k$-simplex and let $i:[k] \to T$ be an enumeration of the vertices compatible with the orientation of $T$, in the sense for instance that $ (\rmd \lambda_{i(1)}, \ldots,  \rmd \lambda_{i(k)})$ is an oriented basis of the cotangent space. One puts:
  \begin{align}\label{eq:lambdat}
 \lambda_T &= k! \sum_{j = 0}^k (-1)^j \lambda_{i(j)} \rmd \lambda_{i(0)} \wedge \ldots  \widehat{\rmd \lambda_{i(j)}} \ldots \wedge \rmd \lambda_{i(k)}.
 \end{align}
The widehat means that the quantity which lies underneath is omitted in the concerned expression.
We remark that:
\beq
\rmd \lambda_T = (k+1)! \, \rmd \lambda_{i(0)} \wedge \ldots \wedge \rmd \lambda_{i(k)}.
\eeq
 
We denote by $\frS[k]$ the set of permutations of $[k]$. The sign of a permutation $\sigma$ is denoted $s(\sigma)$. 

The following result corresponds to Proposition 1 in \cite{Bos02}. 
\begin{proposition} \label{prop:wrec} Pick $k\geq 1$ and let $T$ be a $k$-simplex. We have:
\begin{equation}\label{eq:wrec}
\lambda_T = \sum_{i \in T} \orient(T, T\setminus i) \lambda_i \rmd \lambda_{T\setminus i}.
\end{equation}
\end{proposition}
\begin{proof} We let $i:[k] \to T$ be an enumeration of the vertices compatible with the orientation of $T$. For $j \in [k]$, let $\sigma_j \in \frS[k]$ be the following cyclic permutation:
\begin{equation}
\sigma_j : (0, 1,  \ldots, j, j+1, \ldots k) \mapsto (j, 0, \ldots, j-1, j+1, \ldots k).
\end{equation}
We denote by $t_j$ the terms in (\ref{eq:lambdat}):
\begin{align}
t_j = & (-1)^j \lambda_{i(j)} \rmd \lambda_{i(0)} \wedge \ldots  \widehat{\rmd \lambda_{i(j)}}  \ldots \wedge \rmd \lambda_{i(k)} \\ 
= & s(\sigma_j) \lambda_{i({\sigma_j(0)}) }\rmd \lambda_{i({\sigma_j(1)})} \wedge \ldots \wedge \rmd \lambda_{i({\sigma_j(k)})}.
\end{align}
We also let $i_j : [k-1] \to T \setminus i(j)$ be an enumeration compatible with the orientation of $ T \setminus i(j)$. We let $(j, i_j) : [k] \to T$ denote the enumeration:
\begin{equation}
(0, \ldots, k) \mapsto ( i(j) , i_j(0), i_j(1), \ldots i_j(k-1)).
\end{equation}
We remark that the map $(j, i_j)^{-1} \circ i \circ \sigma_j$ is a permutation of $[k]$, sending $0$ to $0$, so that $i \circ \sigma_j (0) = i(j)$.

In the expression for $t_j$ we may apply the inverse of this permutation to the indices to get:
\begin{align}
 t_j = s ((j, i_j)^{-1} \circ i)    \lambda_{i(j)}  \rmd \lambda_{i_j(0)} \wedge \ldots \wedge \rmd \lambda_{i_j(k-1)} .
\end{align}
From this, identity (\ref{eq:wrec}) follows.
\end{proof}

We can also write:
  \begin{align}
\lambda_T &  = \sum_{\sigma \in \frS[k]}  s(\sigma) \lambda_{i({\sigma(0))} }\rmd \lambda_{i({\sigma(1)})} \wedge \ldots \wedge \rmd \lambda_{i({\sigma(k)})}.
 \end{align}

\begin{proposition}\label{prop:wdof} Let $T$ be a $k$-simplex.
Then, we have:
\beq \label{eq:ltt}
\lambda_T|_T = k! \, \rmd \lambda_{i(1)} \wedge \ldots \wedge \rmd \lambda_{i(k)},
\eeq 
and:
\beq
\int_T \lambda_T = 1.
\eeq
For any  other $k$-simplex $T'\in \calT$ we have $\lambda_T|_{T'}=0$. 
\end{proposition}

\begin{proof} We first claim that on $T$ we have:
\begin{align}
s(\sigma) \rmd \lambda_{i({\sigma(1)})} \wedge \ldots \wedge \rmd \lambda_{i({\sigma(k)})}|_T = \rmd \lambda_{i(1)} \wedge \ldots \wedge \rmd \lambda_{i(k)}|_T.\nonumber
\end{align}
Indeed, if $\sigma(0) = 0$, the identity is clear. If $\sigma(0) = j$ for some $j \in \{1, \ldots, k \}$,  we let  $\tau$ be the permutation exchanging $0$ and $j$ and put $\sigma' = \tau \circ \sigma$. We note that $\sigma'(0) = 0$. Consider the expression:
\begin{align}
a= s(\sigma) \rmd \lambda_{i({\sigma(1)})} \wedge \ldots \wedge \rmd \lambda_{i({\sigma(k)})}|_T. \nonumber
\end{align}
At the position where $\sigma(l)= 0$, we may replace $\rmd \lambda_{i(0)}$ by  $- \rmd \lambda_{i({j})} $.  We get: 
\begin{align}
a  & = - s(\sigma) \rmd \lambda_{i({\sigma(1)})} \wedge \ldots  \wedge \rmd \lambda_{i({j})} \wedge  \ldots \wedge \rmd \lambda_{i({\sigma(k)})}|_T ,\nonumber \\
& =  s(\sigma')  \rmd \lambda_{i({\sigma'(1)})} \wedge \ldots \wedge \rmd \lambda_{i({\sigma'(k)})}|_T ,\nonumber\\
& =  \rmd \lambda_{i(1)} \wedge \ldots \wedge \rmd \lambda_{i(k)}|_T.\nonumber
\end{align}

We deduce:
\begin{align}
\lambda_T|_T & = k! \, (\sum_{j \in T} \lambda_j) \rmd \lambda_{i(1)} \wedge \ldots \wedge \rmd \lambda_{i(k)},\nonumber \\
& = k! \, \rmd \lambda_{i(1)} \wedge \ldots \wedge \rmd \lambda_{i(k)}.\nonumber
\end{align} 
Thus we have obtained (\ref{eq:ltt}). For any  other $k$-simplex $T'$  one of the barycentric coordinates attached to $T$ is $0$ on $T'$, so $\lambda_T|_{T'}=0$.

By Stokes theorem  and an induction reasoning on dimension, we get:
\[ 
\int_T \lambda_T  = \int_T \orient(T,T\setminus i(0) ) \rmd \lambda_{T\setminus i(0)}
= \int_{T\setminus i(0)} \lambda_{T\setminus i(0)}
= 1.
\] 
This concludes the proof.
\end{proof}

Because of Proposition \ref{prop:wdof}, the Whitney forms are linearly independent over $\bbR$. However they are not independent over the ring of polynomials, because of the following identities (see Proposition 2 in \cite{Bos02} or Proposition 3.5 in \cite{RapBos09} for dimension 3, Equation (6.5) in \cite{ArnFalWin09}).
\begin{proposition}\label{prop:wdep}
Let $T$ be a $(k+1)$-simplex.  We have the following relation among Whitney $k$-forms:
\[ 
\sum_{i\in T} \orient(T, T\setminus i) \lambda_i \lambda_{T\setminus i } = 0.
\] 
\end{proposition}
\begin{proof}
Let $a$ denote the left hand side. We use (\ref{eq:wrec}) on each $T \setminus i$, to get:
\begin{align}
a & = \sum_{i,j \in T :  i\neq j }  \orient(T, T\setminus i) \orient(T\setminus i, T\setminus \{ i , j \}) \lambda_i \lambda_j \rmd \lambda_{T\setminus \{ i , j \} } .\nonumber
\end{align}
Here, the contributions of $(i,j)$ and $(j,i)$ cancel two by two, since $\delta \delta = 0$. Therefore $a=0$.
\end{proof}

Then we define the space of Whitney (or Weil \ldots ) forms:
\begin{equation}
\WF^k(\calT)= \lspan \{ \lambda_T \ : \ T \in \calT^k \}.
\end{equation}
Similarly, when $U$ is a simplex in $\calT$, we denote by $\WF^k(U)$ the space of Whitney forms restricted to $U$.

The following result is in \cite{Whi57}.
\begin{proposition}
The exterior derivative maps $\WF^k(\calT)$ to $\WF^{k+1}(\calT)$ and:
\begin{equation}
\rmd \lambda_T = \sum_i \orient( T \cup i , T) \lambda_{T \cup i}.
\end{equation}
\end{proposition}

Following \cite{RavTho77}\cite{Ned80} (for vector fields in $\bbR^2$ and $\bbR^3$ respectively) and \cite{Hip99}\cite{ArnFalWin06} (for differential forms) we now define higher order finite element spaces of differential forms. We adopt the notations of \cite{ArnFalWin06}. Thus for any simplex $T$, $\poly_r\alt^k(T)$ denotes the space of $k$-forms on $T$ which are polynomials of degree $r$. One has:
\begin{equation}
\dim  \poly_r\alt^k(T) = {n+r \choose r}{n \choose k}. \label{eq:dimpr}
\end{equation}

On a vector space $U$ we denote the Koszul operator by $\kappa$. It is the contraction of differential forms on $U$ by the identity on $U$, considered as a vector field. Thus, if $u$ is a $(k+1)$-form on $U$, $\kappa u$ is the $k$-form on $U$ defined by:
\begin{equation}
(\kappa u)_x(\xi_1, \ldots , \xi_k) = u_x(x,\xi_1, \ldots , \xi_k).
\end{equation}
Instead of the Koszul operator (\cite{ArnFalWin06}) one can use the Poincar\'e operator (\cite{Hip99}) associated with the canonical homotopy from the identity to the null-map. 

Define, for any simplex $T$ and any integer $r\geq 1$:
\begin{align}
\poly^-_r\alt^k(T) &=\{u \in \poly_r\alt^k(T) \ : \  \kappa u \in \poly_r \alt^{k-1}(T) \}. 
\end{align}
For wellposedness one should check that these spaces are independent of the choice of origin in $T$ that provides a vector space structure to the affine space of $T$.  As pointed out in \cite{Bos88}, for $r=1$, one recovers Whitney forms:
\begin{align}\label{eq:bos}
\poly^-_{1}\alt^k(T) = \WF^k(T). 
\end{align}

For fixed $r$, the spaces $\poly^-_r\alt^k(T)$ constitute a finite element system which we call the \emph{trimmed} polynomial finite element system of order $r$.  Moreover the use of the Koszul operator guarantees the sequence exactness (\ref{eq:coh}). The Koszul complex also gives the following dimension count, for a simplex $T$ of dimension $n$  (see \cite{ArnFalWin06} equation (3.15)):
\begin{align}
\dim  \poly^-_r\alt^k(T) & = { r + k -1 \choose k}{ n+ r \choose n - k}.\label{eq:dimprm}
\end{align}
The identity (\ref{eq:bos}) can be deduced from this fact.

What remains to be proved, to guarantee that one has a \emph{compatible} finite element system, is the extension property. We could deduce this from the existence of degrees of freedom, proved in \cite{Hip99}\cite{ArnFalWin06}. However the previous  general theorems on finite element systems suggest another organization of the arguments. We also rely heavily on a new tool, namely Proposition \ref{prop:diagdom}. Notice that a much weaker positivity result appeared in the proof of Lemma 4.6 in \cite{ArnFalWin06}. 

\subsection{Compatibility and unisolvence}

For the trimmed polynomial differential forms of order $r$  on simplexes (i.e. $\poly^-_r\alt^k(T)$), the standard system of degrees of freedom is defined as follows. Letting $n$ be the dimension of $T$ we put:
\begin{equation}\label{eq:candof}
\calF^k(T) = \{ u \mapsto \ts \int_T v \wedge u  \  :  \  v \in \poly_{r - n + k -1}\alt^{n -k}(T) \}.
\end{equation}
We will now show that this sysdof is unisolvent on the trimmed polynomial FE system of order $r$, that this element system is compatible, and that its elements are naturally expressed as polynomial multiples of Whitney forms.  These three facts, stated in Theorem \ref{theo:charprl} and Proposition \ref{prop:unisprl}, are intimately connected (at least we have not managed to disentangle them). They should be considered known, but the way we obtain them can in some aspects be considered new. 

The characterizations (\ref{eq:highwhitney}) and (\ref{eq:zerobc}) are detailed in particular in \cite{ArnFalWin09}. Here we obtain them by explicit use of the general tools of FES, such as Proposition \ref{prop:extrec}. We also notice that a streamlined proof of unisolvence appears in \cite{Arn13}. We both check that the third characterization of unisolvence provided in Proposition \ref{prop:unisolve} holds. The main difference is that we rely on Proposition \ref{prop:diagdom} instead of Lemma 3.4 in \cite{Arn13} (which corresponds to Lemma 4.11 in \cite{ArnFalWin06}, see also Lemma 10 in \cite{Hip99}). 

\paragraph{A positivity result.}

We work first on a fixed oriented simplex $U$ of dimension $n$, whose sub-simplexes are all oriented. If $T$ is a sub-simplex of $U$, $\widehat T$ denotes the opposite simplex in $U$. Recall that $\lambda_T$ is the Whitney form associated with simplex $T$, given its orientation.

We denote by $\mu_T$ the scalar function obtained as the product of the barycentric coordinates associated with the vertices of $\widehat T$:
\[ 
\mu_T = \Pi_{i \in \widehat T} \lambda_i.
\] 

Consider a $k$-simplex $T$. Choose an orientation preserving enumeration of its vertices $i: [k] \to T$,  and complete it to an enumeration $i: [n] \to U$, respecting the orientation of $\widehat T$. We have, on $U$:
\begin{align}
\lambda_T \wedge \rmd \lambda_{\widehat T} & = k! (n-k)! \sum_{j \in [k]} (-1)^j \lambda_{i_j} \rmd \lambda_{i_0} \wedge \ldots  \widehat{\rmd \lambda_{i_j}} \ldots \wedge \rmd \lambda_{i_k} \wedge \ldots \wedge \rmd \lambda_{i_n},\nonumber\\
& =  k! (n-k)! \sum_{j \in [k]}  \lambda_{i_j} \rmd \lambda_{i_1} \wedge \ldots \wedge \rmd \lambda_{i_k} \wedge \ldots \wedge \rmd \lambda_{i_n}.\nonumber
\end{align}
Let $s(T)$ be equal to $1$, if $i$ agrees with the orientation of $U$, and $-1$ if not. Thus $s(T)$ depends only on the orientations of $T$, $\widehat T$ and $U$. 
We get:
\[ 
\lambda_T \wedge \rmd \lambda_{\widehat T} = s(T) {n \choose k}^{-1}( \sum_{i \in T} \lambda_i) \lambda_U.
\] 

For any $x \in U$ define a real matrix $D(x)$, indexed by the $k$-dimensional sub simplices $T,S$ of $U$, by:
\[ 
D_{TS}(x) \lambda_U = {n \choose k} s(T) \mu_{T} \lambda_T \wedge \rmd \lambda_{\widehat S}.
\] 
The diagonal is positive.

\begin{proposition}\label{prop:diagdom}
For any $x \in U$, the matrix $D(x)$ is symmetric and weakly diagonally dominant.
\end{proposition}

\begin{proof}
The diagonal terms are of the form:
\[ 
D_{TT}(x) = \mu_{T} \sum_{i \in T} \lambda_i.
\] 

Off diagonal terms $D_{TS} (x)$ can be non zero only for $ \dim (T \setminus S) = 1$. Consider this case and put $T \setminus S = i$.
 We have:
\begin{align}
\lambda_T \wedge \rmd \lambda_{\widehat S} & = (n-k) \orient(T, T\cap S) \orient(\widehat S , \widehat{ T \cup S} )\lambda_i \rmd \lambda_{T \cap S} \wedge \rmd \lambda_i \wedge \rmd \lambda_{\widehat{ T \cup S}} .\nonumber
\end{align}
But we also have, by definition of $s(S)$ :
\begin{align}
\lambda_U 
 & =  s(S)  \frac{n!}{k! (n-k-1)!} \orient(S, T \cap S) \orient(\widehat S , \widehat{ T \cup S} )  \rmd \lambda_{T \cap S} \wedge \rmd \lambda_i \wedge \rmd \lambda_{\widehat{ T \cup S}}.\nonumber
\end{align}
so that:
\begin{align}
\lambda_T \wedge \rmd \lambda_{\widehat S} & = s(S) {n \choose k}^{-1} \orient(T, T \cap S) \orient(S, S \cap T)\lambda_i \lambda_U.\nonumber
\end{align}
It follows that:
\[ 
D_{TS}(x) = s(T)s(S) \orient(T, T \cap S) \orient(S, S \cap T)  \mu_{T} \lambda_i.
\] 

This shows that $D(x)$ is weakly diagonally dominant by rows.

Since also $\mu_{T} \lambda_i = \mu_{T \cup S}$, the matrix $D(x)$ is symmetric.
\end{proof}

We use this result in the following form.
\begin{corollary} For each $x$, the matrix $D(x)$ is symmetric positive semi-definite.
\end{corollary}

\begin{proposition}\label{prop:duality}
Suppose that we have attached to each $k$-simplex $T$ in $U$, a scalar function $u_T$ on $U$. The following are equivalent:
\begin{align}
\sum_T u_T \mu_{T} \lambda_T & = 0,\\
\int_U (\sum_T u_T \mu_{T} \lambda_T) \wedge (\sum_S s(S) u_S \rmd \lambda_{\widehat S} ) & = 0,\\
\sum_S s(S) u_S \rmd \lambda_{\widehat S} & = 0.
\end{align}
Here the summation indices $S$ and $T$ run over the $k$-dimensional sub-simplices of $U$.
\end{proposition}
\begin{proof}
(i) The first condition implies the second.

(ii) Suppose the second condition holds. Since the matrix $D(x)$ is semi positive at each $x$, we get that the integrand is $0$ point-wise. Therefore for each $T$ we have the point-wise equality:
\begin{align}
\mu_{T} \lambda_T \wedge (\sum_S s(S) u_S \rmd \lambda_{\widehat S} ) = 0,
 \end{align}
so that :
 \begin{align}
 \lambda_T \wedge (\sum_S s(S) u_S \rmd \lambda_{\widehat S} ) = 0,
 \end{align}
Since constant differential forms are in the span of the Whitney forms $\lambda_T$, the third condition holds. 
 
 (iii) The reverse implications follow from similar arguments.
\end{proof}


\paragraph{Compatibility and unisolvence.}

The following result is Theorem 4.13 in \cite{ArnFalWin06}, and we follow the proof of Theorem 4.9 of that paper. 
\begin{lemma} \label{lem:dimc} Let $0 \le k \le n$ and $r \ge 1$. Then,
\begin{equation}
\dim \poly^-_{r}\alt^k(U) = \sum_{V \subcell U} \dim \poly_{r - \dim V +k -1}\alt^{\dim V - k}(V) .
\end{equation}
\end{lemma}
\begin{proof} Let $a$ be the left hand side and $b$ the right hand side of the claimed equality.  We write:
\begin{align}
b & = \sum_{l= 0}^n {n+1 \choose l+1} {r+k-1 \choose l} {l \choose k} = \frac{ n+1}{r+k}  {n \choose k } \sum_{l= 0}^n {r+k \choose l+1}{n-k \choose n-l},\\
& =  \frac{ n+1}{r+k}  {n \choose k } {n + r \choose 1+n} = { r + k -1 \choose k}{ n+ r \choose n - k} = a.
\end{align}
The first equality is (\ref{eq:dimpr}) whereas the last is (\ref{eq:dimprm}). From the first to the second line we used the  binomial identity:
\[ 
\sum_{j} {m \choose p-j} {n \choose q+ j} = {m + n \choose p+q}.
\]
The other identities come from elementary manipulations with factorials.
\end{proof}

We adopt the notation:
\[ 
\poly^-_{r}\alt^k_0(U) = \{ u \in  \poly^-_{r}\alt^k(U) \ : \ u|_{\partial U} = 0  \}.
\] 

\begin{theorem}\label{theo:charprl} Let $r\geq 1$. Then, spaces  $\poly^-_{r}\alt^k(U)$ define a compatible finite element system on $\calT$. One has:
\begin{align}\label{eq:highwhitney}
\poly^-_{r}\alt^k(U) = \poly_{r-1}(U) \WF^k(U), 
\end{align}
and, for $n = \dim U$:
\beq \label{eq:zerobc}
\poly^-_{r}\alt^k_0(U)=   \lspan \{u \mu_T \lambda_T\ : \ u \in \poly_{r - n + k -1}(U), T \in \subcell(U)^k \}.
\eeq
Here $\subcell(U)^k$ denotes the set of $k$-dimensional subcells of $U$ and the dependence of $\mu_T$ on $U$ is implicit.
\end{theorem} 

\begin{proof}
By induction on the dimension of maximal simplices of $\calT$.

Suppose that the proposition has been proved for simplicial complexes of dimension $(n-1)$. Let $U$ be a simplex of dimension $n$. Due to the characterization (\ref{eq:zerobc}) applied to each $V \in \partial U$ we may define linear  extension operators:
\[ 
e_V : \poly^-_{r}\alt^k_0(V) \to \poly_{r-1}(U) \WF^k(U),
\] 
such that:
\[ 
(e_V v) |_{V'} = 0  \myfor V' \neq V,\ \dim V' = \dim V.
\] 
The proof of Proposition \ref{prop:extrec} then shows that the restriction operator
\beq
\label{eq:rho}
\rho: \sum_{V\in \partial U}  e_V  \poly^-_{r}\alt^k_0(V)   \to \poly^-_{r}\alt^k(\partial U) ,
\eeq
is onto. In particular the restriction $\poly^-_{r}\alt^k(U) \to \poly^-_{r}\alt^k(\partial U)$ is onto, so that the extension property holds for $U$.

Moreover the sum in (\ref{eq:rho}) is direct and the operator $\rho$ is injective: If $u$ is written
\[ 
u = \sum_{V\in \partial U}  e_V v_V,
\] 
and is mapped to $0$ by $\rho$, then $v_V= 0$ on $k$-dimensional cells $V\in \partial U$, therefore also on $(k+1)$-dimensional ones, etc. 

Let $K$ denote the right hand side in (\ref{eq:zerobc}) and consider the direct sum:
 \[ 
W = K \oplus \bigoplus_{V\in \partial U}  e_V  \poly^-_{r}\alt^k_0(V) ,
\] 
 $W$ is a subspace of  $\poly_{r-1}(U) \WF^k(U)$, which in turn is a subspace of $\poly^-_{r}\alt^k(U)$. By Proposition \ref{prop:duality} applied to each subsimplex of $U$ and Lemma \ref{lem:dimc} we have:
 \begin{align}
\dim W &= \dim K + \sum_{V\in \partial U} \dim \poly^-_{r}\alt^k_0(V), \nonumber\\
  & = \sum_{V\subcell U} \dim \poly_{r - \dim V +k -1}\alt^{\dim V - k}(V), \nonumber\\
  & = \dim \poly^-_{r}\alt^k(U).  \nonumber
\end{align}
Therefore:
\[ 
W = \poly_{r-1}(U) \WF^k(U) = \poly^-_{r}\alt^k(U).
\] 
So (\ref{eq:highwhitney}) holds and (\ref{eq:zerobc}) follows, using the injectivity of $\rho$.
\end{proof}

\begin{proposition}\label{prop:unisprl} The degrees of freedom (\ref{eq:candof}) are unisolvent and define a commuting interpolator.
\end{proposition}
\begin{proof}
The characterization (\ref{eq:zerobc}) and Proposition \ref{prop:duality} show that the integrated wedge product defines an invertible bilinear form on:
\begin{equation}
\poly^-_{r}\alt^k_0(U) \times \poly_{r - n + k -1}\alt^{n -k}(U).
\end{equation}
Proposition \ref{prop:unisolve} and Lemma \ref{lem:dimc} then gives unisolvence.

Commutation follows from Proposition \ref{prop:com}.
\end{proof}

Now we can interpret Proposition \ref{prop:duality} as saying that linear relations in natural spanning families for high order Whitney forms and for their degrees of freedom, correspond to one another.

\subsection{Low order degrees of freedom.}

We have already defined two choices of systems of degrees of freedom for trimmed polynomial differential forms. The canonical one, defined in (\ref{eq:candof}), and the harmonic one, defined in (\ref{eq:harmdofk}, \ref{eq:harmdof}). We now provide a third possibility. These are overdetermining, and we first show how one can deduce a unisolvent sysdof from an overdetermining one.

\paragraph{General construction.}
Suppose $E$ is a compatible finite element system and that $\calF$ is a system of degrees of freedom such that $E_0^k(T) \to \calF^k(T)^\star$ is injective for each $T$ and $k$. This means that an element of $E^k(T)$ is over-determined by the degrees of freedom.

We suppose that $\calF^{\dim T}(T)$ contains the integral. We also suppose that $\rmd^\star: l \mapsto l \circ \rmd$ makes the following sequences exact:
\begin{equation}
\xymatrix{
0 \ar[r] & \bbR \ar[r]&  \calF^{\dim T} (\subcell( T) ) \ar[r]& \ldots \ar[r] & \calF^0(\subcell(T)) \ar[r] & 0.
}
\end{equation}
The first non trivial arrow is inclusion of the integral.

Define spaces $\calF_0^k(T)$ to consist of the forms in $\calF^k(\subcell(T))$ whose restriction to $E_0^k(T)$ are zero. We get exact sequences:
\begin{equation}
\xymatrix{
0 \ar[r] & \calF_0^k(T) \ar[r] &  \calF^k(\subcell(T)) \ar[r] & E_0^k(T)^\star \ar[r] & 0.
}
\end{equation}
The second arrow is inclusion and the third arrow is restriction. We can organize everything in a diagram, with exact rows and columns:
\begin{equation}
\xymatrix{
         &                               & 0 \ar[d]                          & 0 \ar[d]                & 0 \ar[d]          & \\
 & 0   \ar[r] \ar[d]                    & \calF_0^{\dim T} (T) \ar[r] \ar[d]      & \ldots \ar[r] \ar[d] & \calF_0^0(T) \ar[r] \ar[d]       & 0\\
0 \ar[r] & \bbR \ar[r] \ar[d]                 & \calF^{\dim T} (\subcell(T)) \ar[r]  \ar[d]& \ldots \ar[r] \ar[d]  & \calF^0(\subcell(T)) \ar[r] \ar[d] & 0\\
0 \ar[r] & \bbR \ar[r] \ar[d]                  & E_0^{\dim T} (T)^\star \ar[r]   \ar[d]  & \ldots \ar[r]  \ar[d] & E_0^0( T)^\star \ar[r]  \ar[d]    & 0\\
         & 0                                      & 0                                   & 0                    & 0          
}
\end{equation}
The horizontal arrows are almost all $\rmd^\star$.

According to Proposition \ref{prop:subcompl} below it is possible to choose a supplementary $\calG^k(T)$ of $\calF_0^k(T)$ in $\calF^k(\subcell(T))$ in such a way that $\rmd^\star$ maps $\calG^k(T)$ to $\calG^{k-1}(T)$.  It provides a unisolvent system of degrees of freedom with commuting interpolator.

\begin{proposition}\label{prop:subcompl}
Suppose we have a commuting diagram with exact rows and columns:
\begin{equation}
\xymatrix{
         & 0 \ar[d]                               & 0 \ar[d]                      & 0 \ar[d]          & \\
0 \ar[r] & A^0 \ar[r] \ar[d]^{f^0}      & \ldots \ar[r] \ar[d] & A^n \ar[r] \ar[d]^{f^n}        & 0\\
0 \ar[r] & B^0 \ar[r]  \ar[d]^{g^0}     & \ldots \ar[r] \ar[d] & B^n \ar[r] \ar[d]^{g^n}  & 0\\
0 \ar[r] & C^0 \ar[r]   \ar[d]         & \ldots \ar[r]  \ar[d] & C^n \ar[r]  \ar[d]       & 0\\
         & 0                                      & 0                                            & 0          
}
\end{equation}
Then one can choose a subspace $D^k$ of $B^k$ for each $k$, such that:
\begin{equation}
B^k  = f^kA^k \oplus D^k,
\end{equation}
and the differential of $B^\bs$ maps $D^k$ to $D^{k+1}$. Then $g^\bs$ induces an isomorphism of complexes $D^\bs \to C^\bs$.
\end{proposition}
\begin{proof}
Choose a scalar product $a^k$ on $B^k$, denote orthogonality with respect to $a^k$ by $\perp^k$ and define $D^k$ by:
\begin{equation}
D^k = \{ u \in B^k \ : \ u \perp^k \rmd f^{k-1} A^{k-1} \rmins{and} \rmd u \perp^{k+1} \rmd f^k A^k\}.
\end{equation}
This provides one possible choice for $D^k$.
\end{proof}

\paragraph{Small degrees of freedom.} Fix $r\geq 1$. For $E$, take trimmed polynomials of order $r$, that is $E^k(T)= \poly^-_{r}\alt^k(T)$. For $\calF$ take degrees of freedom deduced from its principal lattice of order $r$, $\Sigma_r(T)$ as follows. Let $\Sigma^k_r(T)$ be the so-called small $k$-simplexes in $T$, whose vertices are points in $\Sigma_r(T)$, and which are $1/r$-homothetic to a $k$-face of $T$. Following \cite{RapBos09} we consider integration on elements of $\Sigma^k_r(T)$ as degrees of freedom for $k$-forms, and let $\calF^k(T)$ be the space they span.  We want to show that these degrees of freedom are \emph{over-determining} on $E$.

The proof is based on two lemmas. In the following, $\Pi^k_T$ denotes the interpolation operator onto Whitney $k$-forms, of lowest order, determined by the simplex $T$. Moreover $\tau_\xi$ denotes translation by a factor $\xi$:
\begin{equation}
(\tau_\xi u) (x) = u(x - \xi).
\end{equation}

\begin{lemma}\label{lem:intzero}
Let $V$ be an oriented $n$-dimensional vector space. Let $U$ be an $n$-simplex in $V$. Let $u$ be a polynomial $n$-form on $V$ such that:
\begin{equation}
\forall \xi \in V \quad \int_U \tau_\xi u = 0.
\end{equation}
Then $u = 0$.
\end{lemma} 
\begin{proof}
Granted.
\end{proof}

\begin{lemma}\label{lem:interzero}
Let $V$ be an oriented $n$-dimensional vector space. Let $U$ be a $n$-simplex in $V$. Let $u$ be a $k$-form on $V$ which is polynomial. If $\Pi^k_U \tau_\xi u = 0$ for all $\xi \in V$, then $u=0$.
\end{lemma}
\begin{proof}

For each $k$-face $T$ of $U$ we have $\int_T \tau_\xi u = 0$ for all $\xi \in V$ tangent to $T$. By Lemma \ref{lem:intzero} we get that $u$ pulled back to the tangent space of $T$ is $0$. Since this holds for all translates of $T$, it follows that the pullback of $u$ to any affine $k$-dimensional subspace of $V$ which is parallel to a $k$-face of $U$, is $0$. It follows that for any simplex $T'$ obtained from $T$ by a translation and a scaling, we have $\Pi^k_{T'} u = 0$. At any point $x \in V$ we may choose $T'$ in a sequence shrinking to $x$ and deduce $u(x) = 0$. Therefore $u = 0$.
\end{proof}

\begin{proposition}\label{prop:overdet} Let $V$ be an oriented $n$-dimensional vector space. Let $U$ be a $n$-simplex in $V$. If $u \in \poly^-_{r}\alt^k(U)$ satisfies $\int_{T} u = 0 $ for all $T \in \Sigma^k_r(U)$ then $ u =0$. 
\end{proposition}

\begin{proof}
We may suppose $r >1$, since the case $r= 1$ is standard degrees of freedom for Whitney forms.

For $k = \dim V$ we do the following. Remark that $u$ is polynomial of degree $r-1$. Choose $T \in \Sigma^k_r(U)$ such that the other elements of $\Sigma^k_r(U)$ are obtained as translated by vectors $\xi \in \Sigma_{r-1}(W)$ of some $n$-simplex $W$. Since the map:
\begin{equation}
\xi \mapsto \int_{T} \tau_\xi u,
\end{equation}
is a polynomial of degree $r-1$, which is zero at the points in $\Sigma_{r-1}(W)$, it is zero everywhere. By Lemma \ref{lem:intzero} we have that $u = 0$.

Then we do a descending induction on $k$. Pick then $k < \dim V$ and $u \in \poly^-_{r}\alt^k(U)$ such that $\int_{T} u = 0 $ for all $T \in \Sigma^k_r(U)$. We suppose that the proposition has been proved for elements in $\poly^-_{r}\alt^{k+1}(U)$.

Remark that $ \rmd u \in \poly^-_r\alt^{k+1}(U)$ and that for any $T \in \Sigma^{k+1}_r(U)$, we have $\int_{T} \rmd u = 0$ by Stokes theorem applied on small simplices.  By the induction hypothesis we deduce $\rmd u = 0$. Therefore $u\in \rmd \poly^-_r\alt^{k-1}(U)$ is actually a polynomial of degree $r-1$. For each small $n$-dimensional simplex $S$ in $U$, (i.e. $S \in \Sigma^n_r(U)$) we have $\Pi^k_S u = 0$. Now we consider these $S$ to be translated of one of them, say $S_0$, by factors $\xi$. We have that $\Pi^k_{S_0} (\tau_\xi u)$ is a polynomial of degree $r-1$ in $\xi$, which is $0$ on a regular lattice of weight $r-1$ ($\Sigma_{r-1}$ of some $n$-dimensional simplex, see Figure \ref{fig:nest}), hence it is identically $0$. By Lemma \ref{lem:interzero}, $u = 0$. 
\end{proof}

\begin{figure}
\centering
\includegraphics[width = 6cm]{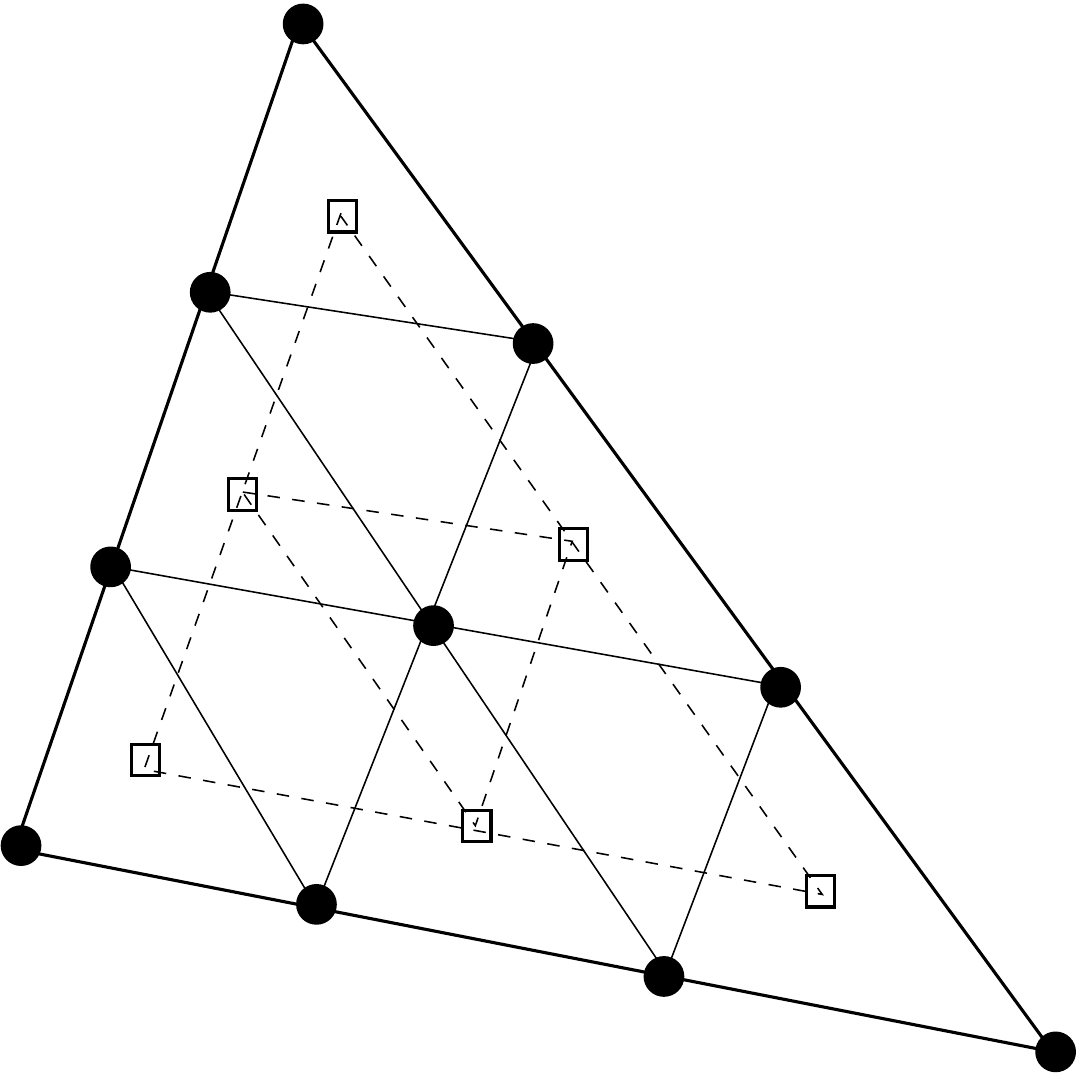}
\caption{A triangle, with the principal lattice of order $3$, and the lattice of order $2$ based on barycenters of small triangles. \label{fig:nest}}
\end{figure}

The construction given in the proof of Proposition \ref{prop:subcompl} depends on a choice of scalar product for each $\calF^k(T)$. The scalar product which makes the canonical basis of $\calF^k(T)$ orthonormal provides a possible choice.

\paragraph{Computation of small degrees of freedom.}
Let $U$ be an $n$-simplex, and $T$ an $k$-face of $U$. We suppose we have fixed an enumeration of $T$ of the form $x:[k] \to T$, which is compatible with the orientation of $T$. We provide a formula for the small degrees of a form $\lambda^\alpha \lambda_T$, with $\alpha \in \Sigma_{r-1}[k]$ and $\lambda^\alpha = \lambda_0^{\alpha_0} \ldots \lambda_k^{\alpha_k}$. Let $T'$ be an oriented $k$-simplex included in $U$. We compute:
\beq \label{eq:smalldof}
\int_{T'} \lambda^\alpha \lambda_T,
\eeq
extending  Proposition 2.3 of \cite{Rap05} to arbitrary dimension. 
Let $x':[k]\to T'$ be an enumeration of the vertices of $T'$ compatible with its orientation.

\begin{proposition}
The form $\lambda_T |_ {T'}$ is constant and:
\beq
\int_{T'} \lambda_T = \det [\lambda_{x(i)} (x'(j))],
\eeq
where the matrix on the right is indexed by $[k]\times [k]$.
\end{proposition}
\begin{proof}
That  $\lambda_T$ is constant on $T'$ follows from the fact that the restriction belongs to $\poly^-_1\alt^k(T')$.

Denote by $M$ the matrix on the right hand side of the equality to prove. Let $\lambda'_{x'(i)}$ denote the barycentric coordinate map on $T'$ attached to vertex $x'(i)$. Given our enumerations, we may denote more simply the coordinate maps attached to $T$, by $\lambda_i$ (defined on U) and those attached to $T'$, by $\lambda'_i$ (defined on $T'$).

We have, on $T'$:
\beq \label{eq:lambdaprime}
\lambda_{i} = \sum_j M_{ij} \lambda'_j.
\eeq
We insert this in the expression:
\begin{align}
\lambda_T & = \sum_{\sigma \in \frS[k]} s(\sigma) \lambda_{\sigma(0)}\rmd \lambda_{\sigma(1)} \wedge \ldots \wedge \rmd \lambda_{\sigma(k)},
\end{align}
and get:
\begin{align}
\sum_{\sigma \in \frS[k]} s(\sigma) \!\!\! \!\!\! \sum_{\sigma'(0), \ldots, \sigma'(k) \in [k]} \!\!\!\!\!\!  M_{\sigma(0) \sigma'(0) }\cdots M_{\sigma(k) \sigma'(k)} \lambda'_{\sigma'(0)} \rmd \lambda'_{\sigma'(1)} \wedge \ldots \wedge \rmd \lambda'_{\sigma'(k)}.
\end{align}
Recall from the proof of Proposition \ref{prop:wdof} that,  when $\sigma'$ is a permutation of $[k]$, we have, on $T'$:
\beq
\rmd \lambda'_{\sigma'(1)} \wedge \ldots \wedge \rmd \lambda'_{\sigma'(k)} = s(\sigma') \rmd \lambda'_{1} \wedge \ldots \wedge \rmd \lambda'_{k}. 
\eeq
Moreover this form has integral $1/k!$. If $\sigma'$ is not a permutation, we may write $\sigma'(i) = \sigma'(j)$ for $i \neq j$, and let $\tau$ be the permutation exchanging $i$ and $j$. Then the terms corresponding to the permutations $\sigma$  and $\tau \circ \sigma$ cancel.

Therefore:
\begin{align}
\int_{T'} \lambda_T & = \frac{1}{(k+1)!} \sum_{\sigma, \sigma' \in \frS[k]} s(\sigma)s(\sigma')M_{\sigma(0) \sigma'(0) }\cdots M_{\sigma(k) \sigma'(k)},\\
&= \sum_{\sigma \in \frS[k]} s(\sigma) M_{\sigma(0) 0 }\cdots M_{\sigma(k) k} .
\end{align}
This proves the proposition.
\end{proof}

To complete the computation of (\ref{eq:smalldof}), insert the expression (\ref{eq:lambdaprime}) in the expression for $\lambda^\alpha$. One obtains a weighted sum of terms of the form $(\lambda')^\beta \lambda_T$ for which one can use:
\beq
\int_{T'} (\lambda')^\beta \lambda_T = \frac{\beta_0 ! \ldots \beta_k ! k!}{(\beta_0 + \ldots +\beta_k + k)!}\int_{T'} \lambda_T.
\eeq

\paragraph{Volumetric interpretation of small degrees of freedom.}

Following \cite{Rap05}, we may prove that small dofs for first order polynomial Whitney $k$-forms  
can be interpreted as volumes of suitable simplices. More precisely, given a $k$-face of a $n$-simplex $U$, we have that
\[  \int_{T'} \lambda_T  = \pm \vol ((U\setminus T)\cup T' )/ \vol( U),\]
for all $k$-simplices $T'$ contained in $U$.  The sign $\pm$ depends on whether the orientations of $T'$ and $T$ agree or not. Here $ (U\setminus T) \cup  T'$ denotes the $n$-simplex whose vertices  
are those of $U\setminus T $ together with those of $T'$. The key argument
of the proof is the recursive formula defining Whitney $k$-forms
starting from Whitney $(k-1)$-forms,  Proposition \ref{prop:wrec}.

A similar interpretation is possible for dofs associated to
high order Whitney $k$-forms. Given two $k$-simplices $T,T'\subset \bbR^n$ in a $n$-simplex $U$
and two multiindices $\alpha, \alpha' \in \Sigma_{r-1} [n]$
 we have
\begin{eqnarray*} 
\int_{\{\alpha',T'\}}  \lambda^{\alpha} \lambda_T
 & = &  \int_{\{\alpha',T'\}} \lambda_T  \int_{\{\alpha',T'\} } \lambda^{\alpha} / \vol \{\alpha',T'\} \\
\end{eqnarray*} 
We have denoted by $ \{\alpha',T'\}$, the small $k$-simplex in $\Sigma^k_r(T)$ obtained as follows. Let $\calB$ be the lattice of barycenters of elements in $\Sigma^n_r (T)$, which is indexed by $\Sigma_{r-1}[n]$. In Figure \ref{fig:nest} the case $r=3$ is represented, with the vertices of $\Sigma_r(T)$ as black bullets and the vertices of $\calB$  as small squares. Then $ \{\alpha',T'\}$ occupies, in the small $n$-simplex with barycenter in $\calB$ determined by $\alpha'$, the analogous position as $T'$ in $T$.


The quantities $ {\rm vol}[ (\{\alpha',T'\}) \vee (U\setminus T) ]$ and thus the 
coefficients of the square matrix $A_{\{\alpha,T\}, \{\alpha',T'\}} =
\langle \lambda^{\alpha} \, \lambda_T, \{\alpha',T'\} \rangle$,
can be evaluated
relying on the Cayley-Menger determinant, which 
allows for computing, in any dimension $n \ge 1$, the volume 
of a $n$-simplex  from the lengths of its sides.

Let us consider a $n$-simplex $T$ with vertices $v_0 , \ldots, v_{n}$ and 
denote by $\ell_{ij}$ the Euclidean distance between $v_i$ and $v_j$. Let us then
define the $[n] \times [n]$ matrix $D$ with entries $D_{ij} = \ell_{ij}^2$, $i,j \in [n]$,
 and
the augmented matrix $\tilde{D}$ 
obtained from $D$ by adding the vector $v = [0,1,1,..., 1]$ as first row and 
$v^\top$ as first column. \\

\begin{proposition}\label{prop:caymen}(Cayley-Menger)
The volume $|T|$ of the $n$-simplex $T$, $n \ge 1$, verifies 
the identity:
\begin{equation}\label{CMdet}
|T|^2 = \frac{(-1)^{n+1}}{2^n \ (n !)^2} \ \det (\tilde{D}).
\end{equation} 
\end{proposition}

\begin{proof}
See \cite{Ber87I}.
\end{proof}

\subsection{Resolutions of trimmed polynomial differential forms}

\paragraph{With boundary condition: $\poly^-_r \alt^{k}_0$.}

Consider a simplex $U$ of dimension $n$ all of whose sub-simplices have been oriented. We will frequently identify a $k$-simplex $T$ with the cochain in $\calC^k$ taking the value $1$ at $T$ and $0$ at other simplices. Also, for a cochain $u$, and a simplex $T$, $u_T$ denotes the value of $u$ at $T$.

For any subsimplex $T$ of $U$ the opposite subsimplex in $U$ is denoted $\widehat T$. Define a sign $s(T)= \pm 1$ as follows. Take an orientation compatible enumeration $[k] \to T$ and an orientation compatible enumeration $[l] \to \widehat T$, with $l+k+1 = n$.  This gives an enumeration $[n] \to U$, and we compare with the orientation of $U$ to define $s(T)$. 
We write:
\beq
s: \mapping{\calC^k(U)}{\calC^{d-k-1}(U)}{T}{s(T)\widehat T}.
\eeq

We denote by $\delta': \calC^{k+1} \to \calC^{k} $ the boundary operator, whose matrix is the transpose of the matrix of the coboundary operator $\delta :\calC^k \to \calC^{k+1} $.
 
\begin{proposition}\label{prop:simphodge}
For the cochains attached to $U$ we have a commuting diagram with exact rows
\beq
\xymatrix{
0 \ar[r] & \bbR \ar[r] \ar[d]^s & \calC^0 \ar[r]^\delta \ar[d]^s &  \ldots  \ar[d]^s  \ar[r]^\delta & \calC^{d-1}  \ar[r]^\delta \ar[d]^s  & \calC^{d}  \ar[r] \ar[d]^s & 0 \\
0 \ar[r] & \calC^d \ar[r]^\deltat & \calC^{d-1} \ar[r]^\deltat & \ldots \ar[r]^\deltat & \calC^{0}  \ar[r]^\deltat  & \bbR  \ar[r] & 0 
}
\eeq
\end{proposition}
\begin{proof}  The identity:
\begin{equation}
s(T) \orient(T,T') = s(T') \orient(\widehat T', \widehat T),
\end{equation}
gives commutation. Exactness for the cochain complex attached to a single simplex is standard.
\end{proof}

We let $\alt^{k}(U)$ denote the space of constant $k$-forms on $U$, which is spanned by the forms $\rmd \lambda_T$, for $(k-1)$-dimensional faces $T$ of $U$. For any  vector space $\calQ$ of functions on $U$ we put:
\begin{equation}\label{eq:qprod}
\calQ \alt^{k}(U)  = \lspan \{uv \ : \ u \in \calQ, \ v \in \alt^{k}(U)\}.
\end{equation}
It is a tensor product in the sense that a free family in $\calQ$ and a free family in $\alt^{k}(U)$, will have products which constitute a free family in $\calQ \alt^{k}(U)$.

\begin{remark}\label{rem:nottensor} Notice that, from a notational point of view, $\poly^-_r \alt^{k}(U)$ does not come from such a construction. There is no space of scalar polynomials $\poly^-_r$ that can be used in the above definition (\ref{eq:qprod}) to obtain $\poly^-_r \alt^{k}(U)$. Moreover, the characterization as a product (see (\ref{eq:highwhitney})):
\beq
\poly^-_r \alt^{k}(U) = \poly_{r-1}(U) \WF^k(U),
\eeq
is not a tensor product. Indeed, Proposition \ref{prop:wdep} shows that there are free families in $\poly_{r-1}(U)$ and in  $\WF^k(U)$, whose product are not free.
\end{remark}


We denote by $\sigma$ the surjection:
\beq
\sigma: \mapping{\calQ \otimes \calC^{k}}{\calQ \alt^{k+1}(U)}{u \otimes T}{u \rmd \lambda_T}
\eeq

\begin{proposition}\label{prop:resdiff}
For any vector space $\calQ$ of functions on $U$ we have resolutions:
\beq
\xymatrix{
\cdots \ \calQ \otimes \calC^{k-2} \ar[r]^\delta & \calQ \otimes \calC^{k-1} \ar[r]^\delta & \calQ \otimes \calC^{k} \ar[r]^\sigma & \calQ \alt^{k+1} \ar[r] & 0.
}
\eeq
\end{proposition}
\begin{proof}
Consider:
\beq \label{eq:resl}
\xymatrix{
\cdots \ \calC^{k-2} \ar[r]^\delta &  \calC^{k-1} \ar[r]^\delta & \calC^{k} \ar[r]^\sigma & \alt^{k+1} \ar[r] & 0.
}
\eeq
We claim that it is an exact sequence. Since $\calC^\bs$ is exact and $\sigma$ surjective, we only need to check the situation at $\calC^k$. We have:
\begin{align}
\sigma u = \sum_T u_T \rmd \lambda_T = \rmd (\sum_T u_T \lambda_T) = \sum_{T'}  (\delta u)_{T'} \lambda_{T'}.
\end{align}
Hence $\sigma \delta = 0$ and  $\sigma u= 0 $ iff $\delta u = 0$, iff $u = \delta v$. This proves the claim.

Taking the tensor product with $\calQ$ proves the proposition. Indeed the product $\calQ \alt^{k+1}(U)$ is a tensor product, which can be written $ \calQ \otimes \alt^{k+1}(U)$.
\end{proof}

We denote by $\sigma'$ the surjection:
\beq
\sigma': \mapping{\poly_q \otimes \calC^k}{\poly^-_r \alt^k_0(U)}{u \otimes T}{u \mu_T\lambda_T},
\eeq
where $q=r-n+k-1$.

\begin{proposition}\label{prop:respl0}
We have resolutions:
\beq
\xymatrix{
\cdots \ \poly_q \otimes \calC^{k+2} \ar[r]^{\deltat} & \poly_q \otimes \calC^{k+1} \ar[r]^\deltat & \poly_q \otimes \calC^{k} \ar[r]^{\sigma'} & \poly^-_r \alt^{k}_0 \ar[r] & 0.
}
\eeq
\end{proposition}
\begin{proof}
Use Proposition \ref{prop:duality} to construct a map $\poly_q \alt^{n-k} \to \poly^-_r \alt^k_0$. Then use Propositions \ref{prop:simphodge}, \ref{prop:resdiff}.
\end{proof}

\paragraph{Without boundary condition: $\poly^-_r \alt^{k}$.}

We define:
\beq
\deltal : \mapping{ \poly_{r-1} \otimes \simp^{k} } { \poly_{r} \otimes \simp^{k-1} }
{u \otimes T} { \sum_{i\in T} \orient (T,T\setminus i) \lambda_{i} u\otimes T\setminus i}
\eeq

and: 
\beq
\beta: \mapping{\poly_{r-1} \otimes \simp^{k}}{\poly^-_{r} \alt^{k-1}}
{u \otimes T}{u\lambda_T}
\eeq

\begin{lemma}\label{lem:seq} We have sequences:
\beq\label{eq:deltaseq}
\xymatrix{
\poly_{r-2} \otimes \simp^{k+1} \ar[r]^\deltal & \poly_{r-1} \otimes \simp^{k} \ar[r]^\deltal & \poly_{r} \otimes \simp^{k-1}
},
\eeq
and:
\beq
\xymatrix{
\poly_{r-2} \otimes \simp^{k+1} \ar[r]^\deltal & \poly_{r-1} \otimes \simp^{k} \ar[r]^\beta & \poly^-_{r} \alt^{k}
}.
\eeq
 \end{lemma}

\begin{proof}
(i) Let $T$ be a $(k+1)$-simplex and $T'$ a $(k-1)$-face of $T$. The two other vertices of $T$ are denoted $i$ and $j$. The component of $\deltal \deltal (u \otimes T)$ on $T'$ is
\beq
(\orient (T, T\setminus i)\orient (T\setminus i, T') + \orient(T, T\setminus j)\orient(T\setminus j, T') )\lambda_i \lambda_j u= 0.
\eeq
Hence $\tau \tau = 0$.

(ii) Let $T$ be a simplex. According to Proposition \ref{prop:wdep} we have
\begin{align}
\sum_{i\in T} \orient(T, T\setminus i) \lambda_i \lambda_{T\setminus i } = 0,
\end{align}
Hence $\beta \tau = 0$.
\end{proof}

The following key observation was made by Bossavit, when comparing the dimension of the space of trimmed polynomial differential forms, with the number of "small" degrees of freedom that determine them (as in Proposition \ref{prop:overdet}). A geometric interpretation of this identity, following \cite{RapBos09},  will be given in the next paragraph.
\begin{lemma}\label{lem:exactdim}
\begin{equation}
\dim \poly^-_r \alt^k = \dim \poly_{r-1} \otimes \simp^k - \dim \poly^-_{r-1} \alt^{k+1}. 
\end{equation}
\end{lemma}
\begin{proof}
In view of (\ref{eq:dimprm}) the identity to prove is:
\begin{align}
{ r + k - 1 \choose k+1}{n + r - 1 \choose n-k-1} + {r + k - 1 \choose k}{n + r \choose n - k} = {n + r -1 \choose r-1}{n+1 \choose k+1}.\nonumber
\end{align}
This follows from an elementary 
computation with factorials.
\end{proof}
The following proposition provides an interpretation of the above dimension count in terms of an exact sequence. 

\begin{proposition} There is a unique map $\alpha$ such that the following diagram commutes:
\beq
\xymatrix{
& \poly_{r-2} \otimes \simp^{k+1}\ar[d]^{\beta} \ar[dr]^\deltal \\
0\ar[r] & \poly^-_{r-1} \alt^{k+1} \ar[r]^\alpha & \poly_{r-1} \otimes \simp^k \ar[r]^\beta & \poly^-_r \alt^k \ar[r] & 0.
}
\eeq
It makes the lower row into an exact sequence.
\end{proposition}

\begin{proof}
We can regard these two assertions as attached to the couple $(r,k)$.

Assume that the assertion has been proved for $(r-1, k+1)$.

(i) The vertical map $\beta$ is onto. Pick $u\in \poly_{r-2} \otimes \simp^{k+1}$ such that $\beta u= 0$. We want to show that $\deltal u = 0$. But we may write, by the induction hypothesis $u= \alpha v$. Moreover $v= \beta w$. Then $\deltal u = \deltal \deltal w = 0$. This shows existence and uniqueness of $\alpha$.

(ii). We know that $\beta$ is onto. Since $\beta \deltal = 0$ we have $\beta \alpha = 0$, so we have a sequence. We now prove injectivity of 
$\alpha$. Pick $u\in \poly^-_{r-1} \alt^{k+1} $ and suppose $\alpha u = 0$. Write
\beq
u = \beta v \quad \textrm{with} \quad  v = \sum_T v_T \otimes T.
\eeq 
We know that $\deltal v = \alpha \beta v = \alpha u = 0$.
We have:
\begin{align}
\beta v & = \sum_T v_T \lambda_T,\\
             & = \sum_T \sum_{i \in T} v_T \orient(T, T \setminus i) \lambda_i \rmd \lambda_{T \setminus i}.
\end{align}
In this sum we recognize, in front of $\rmd \lambda_{T \setminus i}$, the component of $\deltal v$ on $T \setminus i$, which is $0$. Hence $u = \beta v = 0$.

The only remaining point to check now is exactness of the sequence. This follows from Lemma \ref{lem:exactdim}.
\end{proof}

\begin{proposition}\label{prop:respl} We have a resolution:
\beq
\xymatrix{
\cdots \ \poly_{r-3} \otimes \simp^{k+2} \ar[r]^\deltal & \poly_{r-2} \otimes \simp^{k+1} \ar[r]^\deltal & \poly_{r-1} \otimes \simp^{k} \ar[r]^\beta & \poly^-_{r} \alt^{k} \ar[r] & 0.
}
\eeq
\end{proposition}

\begin{proof}  In view of Lemma \ref{lem:seq} we have a sequence, so that only exactness needs to be proved.

 (i) We first prove exactness of the composition $\beta \tau$. Suppose $\beta u = 0$. Then we can write $u = \alpha v$ and $v = \beta w$. This gives $u = \tau w$.

(ii) Next we prove exactness of the composition $\deltal\deltal$.  Suppose $\deltal u = 0$. Then $\alpha \beta u = 0$, so $\beta u = 0$ by injectivity of $\alpha$. Therefore we can write $u = \alpha v$. Moreover $v= \beta w$. Then $u = \deltal w$.
 \end{proof}

Recall that in \cite{ArnFalWin09} the essential ingredient to deduce bases from spanning families of highorder Whitney forms is Lemma 4.2 of \cite{ArnFalWin06}. We state this result in the next proposition and show how it can be deduced from Proposition \ref{prop:respl}.

\begin{proposition}
Let $U$ be a simplex, in which we choose a distinguished vertex $O$. The Whitney forms on $U$ pertaining to subsimplices of $U$ containing $O$ are linearly independent over the ring of polynomials on $U$.
\end{proposition}
\begin{proof}
Suppose we have polynomials $u_T$ for each $k$-face $T$ of $U$, such that $u_T=0$ if $T$ does not contain $O$. Suppose that:
\begin{equation}
\sum_T u_T \lambda_T = 0.
\end{equation}
According to Proposition \ref{prop:respl} this can happen iff for each $(k-1)$-face $T'$ of $U$ we have:
\begin{equation}
\sum_T \orient(T,T') \lambda_{T \setminus T'} u_T = 0.
\end{equation}
But when $T'$ does not contain $O$, there is only one term contributing to this sum, namely $T= T' \cup O$. Then $u_T = 0$. Since all $T$ containing $O$ arise in this way we are done.
\end{proof}

\paragraph{Geometric interpretation of Lemma \ref{lem:exactdim}.}

This identity has been firstly observed geometrically by Bossavit,
already before the preparation of \cite{RapBos09}. Indeed, 
 to determine the dimension of the space $\poly_r^-\alt^k (T)$ attached  to an $n$-simplex $T$, we may proceed as follows.
 
Let us start with $n=2$ and consider  the $n$-simplex $T$ together with the 
principal lattice of degree $r > 1$ (see Figure \ref{dfrag2}). 
Connecting the points of this lattice
by lines parallel to the sides of $T$, 
one obtains a partition of $T$ consisting of
$n$-simplexes homothetic to $T$ (the ``small" $n$-simplexes, as they have been called in \cite{RapBos09})
and  other shapes (the so-called ``holes'' in \cite{RapBos09}).

In Figure \ref{dfrag2} (center) the holes are the reversed
triangles  such as the one represented with a thick boundary. Apply a fragmentation operation, cutting $T$ along the lines,  keep the small $n$-simplexes
and eliminate 
all the holes (Fig. \ref{dfrag2} right).

\begin{figure}
\centering
\includegraphics[width = \textwidth]{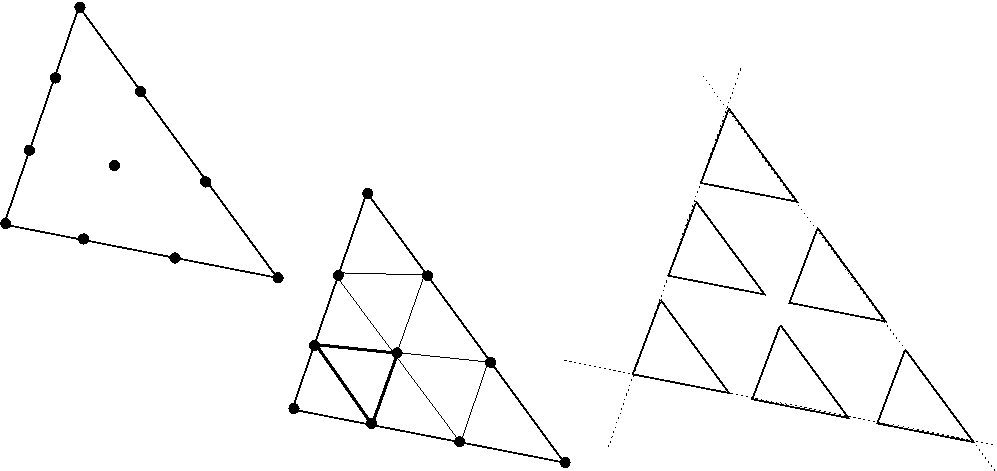}
\caption{A $2$-simplex and the principal lattice of order $r = 3$ (left), the small edges (center), and the fragmented configuration (right).  }
\label{dfrag2}
\end{figure}

The value of $\dim \poly_{r-1}\otimes \calC^k$ corresponds
to the cardinality of the set of small $k$-simplices in the fragmented configuration.
For $r=3$, in Figure \ref{dfrag2} (right), we have 18
nodes, 18 small edges and 6 small triangles. The value of $ \dim \poly_{r-1}^- \alt^{k+1}$
is the number of relations between $k$-simplices when we wish to 
connect the small $n$-simplices
of the fragmented configuration in a lattice structure as in Figure \ref{dfrag2} (center).
For $r=3$, we need to impose 8 conditions among nodes and 3 conditions among edges
(one condition for each reversed triangle, as explained in \cite{RapBos09}).
The value of  $\dim \poly_r^- \alt^k$  is  $\dim \poly_{r-1}\otimes \calC^k$,
(that is the number 
of $k$-simplices in the fragmented configuration),  minus 
$\dim \poly_{r-1}^- \alt^{k+1}$, that is the number of relations 
(see Proposition 3.5 in \cite{RapBos09}) among small $k$-simplices.

These dimension counts are tabulated in Table \ref{tabd2}.

 \begin{table}
\centering
\begin{tabular}{|c|c|c|c|c|}\hline
$r$ & $k$& $\dim \poly_{r-1}\otimes \calC^k$   & 
$ \dim \poly_{r-1}^- \alt^{k+1}$ & $\dim \poly_r^- \alt^k$ \\ \hline

$1$ & 0 & 3 & 0 & 3 \\
& 1 & 3 & 0 & {\bf 3} \\
& 2 & 1 & 0 & {\bf 1} \\ \hline
$ 2$ & 0 & 9  & {\bf 3} & 6 \\
& 1 & 9  & {\bf 1} & {\tt 8} \\
& 2 & 3 & 0 & {\tt 3} \\\hline
$3$ & 0 & 18 & {\tt 8}  & 10 \\
& 1 & 18 & {\tt 3}  & {\em 15} \\
& 2 & 6 & 0 & {\em 6} \\\hline
$ 4$ & 0 & 30 & {\em 15} & 15 \\
& 1 & 30 & {\em 6}  & 24 \\
& 2 & 10 & 0 & 10\\ \hline
$ ... $ & ... &... & ...& ... \\ \hline
\end{tabular}

\caption{$n=2$ in a $n$-simplex $T$.
  \label{tabd2}}
  
\end{table}

\bigskip

A similar counting can be performed for $n=3$ (see Figure \ref{dfrag3}).
Let us consider  the $n$-simplex $T$ together with the 
principal lattice of degree $r > 1$ (In Figure \ref{dfrag3}, $r = 3$). 
Connecting the points of this lattice
by planes parallel to the faces of $T$, 
one obtains a partition of $T$ including 
$n$-simplexes homothetic to $T$ (the ``small" $n$-simplexes)
and other objects (the ``holes'') that can only be octahedra and reversed tetrahedra, \emph{e.g.}, the objects with thick boundary in Fig. \ref{dfrag3} (center).

Apply a fragmentation operation cutting $T$ along the planes,
keep the small $n$-simplexes
and eliminate 
all the holes (Figure \ref{dfrag3}, right). 
Again, the value of  $\dim \poly_{r-1}\otimes  \simp^k$ corresponds
to the cardinality of the set of small $k$-simplices in the fragmented configuration.

For $r=3$, in Figure \ref{dfrag3} right, we have 40
nodes, 60 small edges,  40 small faces and 10 small tetrahedra. 
The value of $\dim \poly_{r-1}^- \alt^{k+1}$
is the number of relations between $k$-simplices when we wish to 
connect the small $n$-simplices
of the fragmented configuration in a lattice structure as in Figure \ref{dfrag3} center.
For $r=3$, we need to impose 20 conditions among nodes, 15 conditions among edges
(12 conditions, one for each reversed triangle on the faces of $T$,
 plus 3 conditions, instead of 4,
for the faces of the central 
reversed tetrahedron), 4 conditions among faces (one for each octahedron).
The value of $\dim \poly_r^- \alt^k$ is  equal to $\dim \poly_{r-1} \otimes  \simp^k$,
that is the number 
of $k$-simplices in the fragmented con\-fi\-gu\-ra\-tion,  minus 
$\dim \poly_{r-1}^- \alt^{k+1}$ , that is the number of relations among small 
$k$-simplices, as given in Table \ref{tabd3}. 

\begin{figure}
\centering
\includegraphics[width = \textwidth]{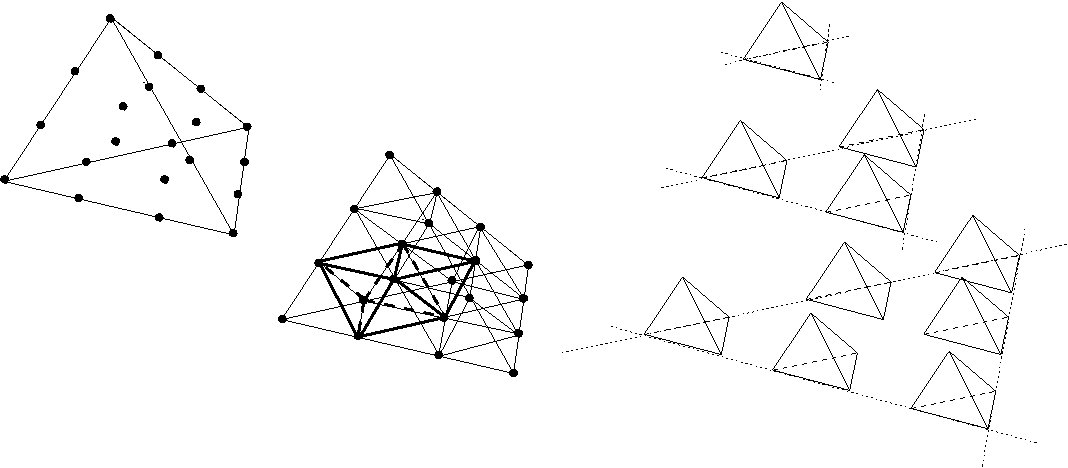}
\caption{A $3$-simplex and the principal lattice of order $r = 3$ (left), the small edges (center) and the fragmented configuration of small tetrahedra (right).  }
\label{dfrag3}
\end{figure}

\begin{table}
\centering

\begin{tabular}{|c|c|c|c|c|}\hline
$r$ & $k$& $\dim \poly_{r-1}\otimes  \simp^k$   & 
$ \dim \poly_{r-1}^- \alt^{k+1}$ & $\dim \poly_r^- \alt^k$  \\ \hline

$1$ & 0 & 4 & 0 & 4 \\
& 1 & 6 & 0 & {\bf 6} \\
& 2 & 4 & 0 & {\bf 4} \\
& 3 & 1 & 0 & {\bf 1} \\\hline
$ 2$ & 0 & 16 & {\bf 6} & 10 \\
& 1 & 24 & {\bf 4} & {\tt 20} \\
& 2 & 16 & {\bf 1} & {\tt 15} \\
& 3 & 4 & 0 & {\tt 4} \\\hline
$3$ & 0 & 40 & {\tt 20} & 20 \\
& 1 & 60 & {\tt 15} & {\em 45} \\
& 2 & 40 & {\tt 4} & {\em 36} \\
& 3 & 10 & 0 & {\em 10} \\\hline
$4$ & 0 & 80 & {\em 45} & 35 \\
& 1 & 120 & {\em 36}  & 84 \\
& 2 & 80  & {\em 10} & 70 \\
& 3 & 20 & 0 & 20 \\\hline
$ ... $ & ... &... & ...& ... \\ \hline
\end{tabular}

\caption{$n=3$ in a $n$-simplex $T$  \label{tabd3}}
\end{table}

\section{Various computations\label{sec:various}}

In this section we present some computations with trimmed polynomial differential forms, expressed as polynomial multiples of Whitney forms. We first provide a parametrized family of bases for polynomials on simplices, which contains the Bernstein and the Lagrange basis as special cases, and allows for de Casteljau type evaluations. Then we show how to compute scalar products of trimmed polynomial differential forms, from the edge lengths of simplices. Finally we provide a formula for wedge products.

\subsection{Polynomials on a simplex \label{sec:polysimp}}

Let $T$ be a simplex of dimension $k$, with vertices enumerated by a map $[k] \to T$. The barycentric coordinates are then denoted $\lambda_0, \ldots, \lambda_k$. For any integer $r \geq 1$, let $\Sigma_r [k]$ be the set of multi-indices $\alpha: [k] \to \bbN$ with weight $|\alpha|= \alpha_0 + \ldots + \alpha_k = r$. It corresponds  to the so-called principal lattice of order $r$ of $T$, consisting of the points in $|T|$ with barycentric coordinates $\alpha/r$, $\alpha \in \Sigma_r [k]$.

The Bernstein basis of $\poly_r(T)$ is indexed by $\Sigma_r [k]$. For $\alpha \in \Sigma_r  [k]$ the attached basis function is:
\begin{equation}
B^\alpha = \frac{r!}{\alpha_0! \ldots \alpha_k !} \lambda_0^{\alpha_0} \ldots \lambda_k^{\alpha_k} .
\end{equation}

The following is a recipe for finding alternative bases of $\poly_r(T)$ indexed by $\Sigma_r [k]$. Fixing $r$, choose, for each $i \in [k]$ and $j \in [ r]$,  a polynomial $\beta_i[j]$ on $\bbR$ of degree $j$. Given such a family $\beta$,  define,  on $T$, for each $\alpha \in \Sigma_r [k]$ the polynomial of degree $r$:
\begin{equation}
C^\alpha = \beta_0[\alpha_0](\lambda_0)\ldots \beta_k[\alpha_k](\lambda_k).
\end{equation}
Specializing further the construction, we choose for each $i \in [k]$, points $t_i[j]$ in $\bbR$  for $0 \leq j < r$ and define $\beta_j$ as follows. We put $\beta_i[0] = 1$ and for $0 < j \leq r$ define the polynomial function:
\begin{equation}
\beta_i[j] : t \mapsto (t-t_i[0]) \ldots (t- t_i[j-1]),
\end{equation}
We then consider the family $C^\alpha$ for $\alpha \in \Sigma_r[k]$. We remark that,  up to scalar factors:
\begin{itemize}
\item the Bernstein basis corresponds to defining $\beta_i[j]$ with $t_i[j] = 0$,
\item the Lagrange basis corresponds to defining $\beta_i[j]$ with $t_i[j] = j/r$.
\end{itemize}

We impose that for all multi-indices $\alpha$ such that $|\alpha|< r$ we have:
\begin{equation}\label{eq:cond}
\sum_{i \in [k]} t_i[\alpha_i] \neq 1.
\end{equation}
This is trivially the case for Bernstein polynomials. More generally, if $t_i[j] \leq j/r$ we get, whenever $|\alpha | < r$:
\[ 
\sum_{i \in [k]} t_i[\alpha_i] \leq |\alpha|/r < 1.
\] 
In particular, the Lagrange basis also satisfies (\ref{eq:cond}).

\begin{remark} If condition (\ref{eq:cond}) does not hold, define a multi-index $a$ with $|a| < r$, such that $\sum_i t_i[a_i] = 1$. Let $x\in T$ be the point with barycentric coordinates $t_i[a_i]$. For all $\alpha \in \Sigma_r [k]$ there is an $i \in [k]$ such that $a_i < \alpha_i$ and then $\beta_i[\alpha_i](t_i[a_i]) = 0$ so that $C^\alpha(x) = 0$. Hence all elements in the span of the $C^\alpha$, annihilate $x$. In particular the span does not contain the constant polynomials.
\end{remark}

\begin{proposition} Suppose condition (\ref{eq:cond}) holds.
Then the family $C^\alpha$, $\alpha \in \Sigma_r[k]$ is a basis for  $\poly_r(T)$.
\end{proposition} 
\begin{proof}
Let $\calQ_r$ be the span of $C^\alpha$ with  $\alpha \in \Sigma_r [k]$. 

Given $r$, we suppose the proposition holds for $r-1$. We first show that $Q_r$ contains $\poly_{r-1}(T)$.

For any $u \in \poly_{r-1}(T)$, choose coefficients $c_\alpha$ for $\alpha \in \Sigma_{r-1}[k]$ so that:
\[ 
\sum_\alpha c_\alpha C^\alpha = u.
\] 
We regard $t \in \bbR[t]$ as a polynomial. We denote by $1_i \in \Sigma_1 [k]$ the tuple with value $1$ at position $i$, and  value $0$ elsewhere. We write:
\begin{equation}\label{eq:creq}
C^{\alpha + 1_i} = C^\alpha( t-t_i[\alpha_i]) (\lambda_i) = C^\alpha(\lambda_i -t_i[\alpha_i]),
\end{equation}
therefore:
\[ 
\sum_{i \in [k]}  C^{\alpha + 1_i} = C^\alpha (1 - \sum_{i \in [k]} t_i[\alpha_i] ).
\] 
For each $\alpha \in \Sigma_{r-1} [k]$, put:
\[ 
d_\alpha =(1 - \sum_{i \in [k]} t_i[\alpha_i] )^{-1}.
\] 
Then we have:
\[ 
\sum_{i \in [k]} \sum_{\alpha \in \Sigma_{r-1}[k]}  c_\alpha d_\alpha C^{\alpha + 1_i}  = \sum_{\alpha \in \Sigma_{r-1}[k]}  c_\alpha C^\alpha
 = u.
\] 
This shows that $\calQ_r$ contains $\poly_{r-1}(T)$.

Next we remark that for any $u\in \poly_{r-1}(T)$, written as above, and any $i \in [k]$:
\[ 
\lambda_i u = \sum_\alpha c_\alpha C^\alpha( t-t_i[\alpha_i]) (\lambda_i)  + \sum_\alpha c_\alpha t_i [\alpha_i] C^\alpha.
\] 
Therefore $\calQ_r$ also contains $ \lambda_i u$.

This proves that $\calQ_r$ equals $\poly_{r}(T)$.  Dimension count then shows that the $C^\alpha$ with  $\alpha \in \Sigma_r [k]$ constitute a basis. 

\end{proof}

The de Casteljau algorithm is an important algorithm to evaluate Bernstein polynomials. It has received attention recently in finite element contexts \cite{Kir14}\cite{Ain14}. It can be extended to the above bases as follows. We keep notations of the above proof. We first rescale $C^\alpha$ as in the Bernstein basis:
\begin{equation}
\tilde C^{\alpha} = \frac{r!}{\alpha_0! \ldots \alpha_k !} C^{\alpha}.
\end{equation}
Since:
\begin{equation}
\frac{r!}{\alpha_0! \ldots \alpha_k !} = \sum _{i \ : \ \alpha_i \geq 1} \frac{(r-1)!}{\alpha_0! \ldots (\alpha_i -1)! \ldots \alpha_k !},
\end{equation}
we deduce, using (\ref{eq:creq}):
\begin{equation}
\tilde C^{\alpha} (x) = \sum_{i \ : \ \alpha_i \geq 1} \tilde C^{\alpha - 1_i}(x) (\lambda_i(x) - t_i[\alpha_i-1]).
\end{equation}
Therefore, given coefficients $(c_\alpha)$ for $\alpha \in\Sigma_r[k]$ and a point $x$, we can determine coefficients $(c_\alpha)$ for $\alpha \in\Sigma_{r-1}[k]$ such that:
\begin{equation}\label{eq:cast}
\sum_{|\alpha| = r} c_\alpha \tilde C^{\alpha}(x) = \sum_{|\alpha| = r-1} c_{\alpha} \tilde C^{\alpha}(x),
\end{equation}
by putting, for $\alpha \in\Sigma_{r-1}[k]$:
\begin{equation}
c_{\alpha} = \sum_i (\lambda_i(x) - t_i[\alpha_i]) c_{\alpha + 1_i}.
\end{equation}
Repeting the procedure for decreasing $r$, we end up with just one coefficient $c_0$, and this is the value of (\ref{eq:cast}). This provides a stable way of evaluating polynomials expressed in these bases.

\subsection{Computation of scalar products}

Let $V$ be a finite dimensional real vector space. A scalar product on $V$ is a bilinear form $g$ on $V$ which is symmetric and positive definite.
Suppose that we have a spanning family $(e_i)_{i \in I}$, identified with a surjection $\epsilon:\bbR^I \to V$.  We suppose that we know the numbers:
\begin{equation}
G_{ij}= g(e_i, e_j).
\end{equation}
By linear algebra techniques one can construct an $I \times J$ matrix $B$ with independent columns, spanning the kernel of $G$, which is also the kernel of $\epsilon$. Then we have an exact sequence:
\beq
\xymatrix{
0 \ar[r] &\bbR^J \ar[r]^B & \bbR^I \ar[r]^\epsilon & V \ar[r] & 0.
}
\eeq
This provides an alternative to the computation with resolutions used previously in this paper. It is also of independent interest to compute scalar products of differential forms, since variational formulations of PDEs make them appear.

The scalar product $g$ induces a scalar product on $V^\star$ by requiring that $u\mapsto g(u, \cdot)$ is an isometry. In other words if $l, l' \in V^\star$ are represented as $l= g(u, \cdot)$ and $l'=g(u',\cdot)$ then we define $g(l,l')= g(u,u')$. There is a unique scalar product on multi-linear forms such that:
\begin{equation*}
g(u_1 \otimes \cdots \otimes u_k, v_1 \otimes \cdots \otimes v_k)= g(u_1,v_1)\cdots g(u_k,v_k).
\end{equation*}
If we restrict this scalar product to alternating forms we notice:
\begin{equation*}
g(u_1 \wedge \cdots \wedge u_k, v_1 \wedge \cdots \wedge v_k)= k! \sum_{\sigma} s(\sigma) g(u_1,v_{\sigma_1})\cdots g(u_k,v_{\sigma_k}).
\end{equation*}
Actually one scales this scalar product and defines:
\begin{equation*}
g(u_1 \wedge \cdots \wedge u_k, v_1 \wedge \cdots \wedge v_k)= \det [g(u_i, v_j)]_{ij}.
\end{equation*} 
With this scaling, if $(e_i)_{1 \leq i \leq n}$ is an orthonormal basis of $V^\star$, then the family:
\begin{equation*}
e_{i_1} \wedge \cdots \wedge e_{i_k} \rmins{for} 1 \leq i_1 < \cdots < i_k \leq n, 
\end{equation*}
indexed by subsets of $\{1, \cdots, n \}$ of cardinality $k$, is orthonormal in $\Alt^k(V)$.


Let $T$ be a simplex of dimension $k$. Denote the barycentric coordinates by $\lambda_i$. Recall the formula:
\begin{equation}\label{eq:intlambda}
\int_T \lambda_0^{\alpha_0} \cdots \lambda_k^{\alpha_k} = \frac{\alpha_0! \cdots \alpha_k ! k! }{(\alpha_0 + \cdots + \alpha_k +  k)!} \vol T.
\end{equation}

It follows that scalar products of polynomial differential forms on $T$, can be determined from the knowledge the edge lengths of $T$, using two intermediate results:
\begin{itemize}
\item The volume of $T$, as given by the Cayler Menger determinant.
\item The scalar products of $\rmd \lambda_i \cdot \rmd \lambda_j$.
\end{itemize}
For the Cayley Menger determinant see Proposition \ref{prop:caymen}. Here we consentrate on $\rmd \lambda_i \cdot \rmd \lambda_j$.

The tangent space of $T$ is denoted $V$ and the vertices of $T$ are denoted $x_i$.
We denote:
\[ 
y_{ji} = x_j - x_i.
\] 
We have, for $j \neq i$:
\[ 
\rmd \lambda_i (y_{ji}) = -1,
\] 
and, when $i \not \in \{j, k\}$:
\[ 
\rmd \lambda_i (y_{jk}) = 0.
\] 
As developed in Regge calculus \cite{Chr04M3AS}\cite{Chr11NM}, for a constant metric $g$ on $T$ there is a unique symmetric matrix $[g_{ij}]$ such that:
\[ 
g = - \frac{1}{2} \sum_{i \neq j} g_{ij} \rmd \lambda_i \otimes \rmd \lambda_j.
\] 
The edge lengths are then given by:
\[ 
|y_{ji}|^2 = y_{ji} \cdot y_{ji} = g(y_{ji}, y_{ji}) = g_{ij}.
\] 
For each $i$ we determine a vector $z_i$ such that:
\[ 
\forall \xi \in V \quad \rmd \lambda_i (\xi) = g(z_i, \xi).
\] 
We write:
\[ 
z_i = \sum_{j \neq i} z_{ji} y_{ji},
\] 
and determine the scalar coefficients $z_{ji}$ by:
\[ 
\forall k \neq i \quad g(z_i, y_{ki}) = \rmd \lambda_i (y_{ki}) = -1.
\] 
In other words:
\[ 
\forall k \neq i \quad \sum_{j \neq i} z_{ji} g(y_{ji}, y_{ki}) = -1.
\] 
This is a linear system for the vector $(z_{ji})$, with matrix coefficients associated with $j \neq i,\ k\neq i$ given by:
\[ 
g(y_{ji}, y_{ki}) = g_{ji} + g_{ki} - g_{jk}, 
\] 
a formula which can be obtained from:
\[ 
|y_{jk}|^2 = |y_{ji}|^2 + |y_{ki}|^2 - 2 y_{ji} \cdot y_{ki},
\] 
and:
\begin{align}
y_{ji} \cdot y_{ki} & = - \frac{1}{2} \sum_{j' \neq k'} g_{j'k'} \rmd \lambda_{j'}(y_{ji}) \otimes \rmd \lambda_{k'}(y_{ki}),\nonumber \\
& = - \frac{1}{2} g_{jk}.
\end{align}
With these notations one obtains:
\[ 
 \rmd \lambda_i \cdot \rmd \lambda_j = \rmd \lambda_j (z_i) = z_{ji}.
\] 

\subsection{Computation of wedge products}

The following is extracted from the preprint of \cite{Chr07NM}. The proof provides a formula for the wedge product in barycentric coordinates.

\begin{proposition}\label{prop:wedge}
For each $k,l$ and $p,q$ the wedge of forms provides a map:
\begin{equation}
\wedge : \poly^-_{r}\alt^k \times     \poly^-_{q}\alt^l        \to \poly^-_{r+q}\alt^{k+l}.
\end{equation}
\end{proposition}
\begin{proof}
We prove that $\wedge : \WF^k \times \WF^l \to \poly_1\WF^{k+l}$,
with $\WF^k$ defined in (\ref{eq:bos}),
 from which the proposition follows immediately.

Let $(u_i)$ be some family of functions indexed by consecutive integers. For any set of consecutive integers $k < \cdots < l$ we put:
\begin{equation}
\delta u_{[k \cdots l]}= \rmd u_k \wedge \cdots \wedge \rmd u_l.
\end{equation}
This notation will also be used when one index is missing in the set $\{k, \cdots, l \}$, e.g.:
\begin{equation}
\delta u_{[k \cdots \hat{i} \cdots l]}= \rmd u_k \wedge \cdots 
\widehat{(\rmd u_i)} \cdots \wedge \rmd u_l.
\end{equation}
We also put:
\begin{equation}
u_{[k \cdots l]}= \sum_{i=k}^l (-1)^{i-k} u_i \delta u_{[k \cdots \hat{i} \cdots l]}, 
\end{equation}
a notation which is extended straightforwardly to the case of one missing index.

We will prove, by induction on $k$, that:
\begin{equation}
u_{[0 \cdots k-1]} \wedge u_{[k \cdots k+l]} = (-1)^{k-1} \sum_{i=0}^{k-1} (-1)^i u_i u_{[0 \cdots \hat{i} \cdots k+l]}.
\end{equation}
It is evidently true for $k = 1$ and, if it is true for a given $k \geq 1$, we can make the following computations. We remark that:
\begin{equation}
u_{[-1 \cdots k-1]} \wedge u_{[k \cdots k+l]} = (u_{-1} \delta u_{[0 \cdots k-1]} - \delta u_{-1} \wedge  u_{[0 \cdots k-1]})\wedge u_{[k \cdots k+l]}.
\end{equation}
Concerning the first term on the right hand side, we see that:
\begin{equation}\label{eq:firstpart}
u_{-1} \delta u_{[0 \cdots k-1]}\wedge u_{[k \cdots l]} = (-1)^k u_{-1}(u_{[0 \cdots k+l]} - \sum_{i=0}^{k-1}u_i \delta u_{[0 \cdots \hat{i} \cdots k+l]}).
\end{equation}
For the second term we remark that (by the induction hypothesis):
\begin{eqnarray}
& & \delta u_{-1} \wedge  u_{[0 \cdots k-1]}\wedge u_{[k \cdots k+l]}\\
 &=& (-1)^{k-1}\delta u_{-1} \wedge \sum_{i=0}^{k-1} (-1)^i u_i u_{[0 \cdots \hat{i} \cdots k+l]}\\
\label{eq:secondpart}
&=& (-1)^{k-1}\sum_{i=0}^{k-1} (-1)^i u_i (-u_{[-1 \cdots \hat{i} \cdots k+l]} + u_{-1} \delta u_{[0 \cdots \hat{i} \cdots k+l]})
\end{eqnarray}
Now the last term in (\ref{eq:firstpart}) cancels with the last term in (\ref{eq:secondpart}) and we are left with:
\begin{equation}
u_{[-1 \cdots k-1]} \wedge u_{[k \cdots k+l]} = (-1)^k ( u_{-1} u_{[0 \cdots k+l]} - \sum_{i=0}^{k-1} (-1)^i u_i u_{[-1 \cdots \hat{i} \cdots k+l]}).
\end{equation}
This completes the induction and hence the proof.
\end{proof}



\section*{Acknowledgements}
The authors are grateful to Ragnar Winther for stimulating conversations on finite element exterior calculus and to Alain Bossavit for his enlightenments on the geometry of Whitney forms. 

The problem of extending the de Casteljau algorithm in \S \ref{sec:polysimp} was suggested to us by Michael Floater.

SHC thanks Laboratoire Jean Alexandre Dieudonn\'e in Nice, for the invitation in 2010, where this work was started, and Laboratoire Jacques-Louis Lions in Paris,  for the kind hospitality in 2013, where it was completed.

SHC was supported by the European Research Council through the FP7-IDEAS-ERC Starting Grant scheme, project 278011 STUCCOFIELDS.

\bibliography{alexandria,mybibliography}{}
\bibliographystyle{plain}

\end{document}